\documentclass{amsart}

\title[$K_0$ OF SEMIARTINIAN VON NEUMANN
REGULAR RINGS]
{$\mathbf{K}_0$ OF SEMIARTINIAN VON NEUMANN
REGULAR RINGS. DIRECT FINITENESS VERSUS UNIT-REGULARITY. }

\date{July 09, 2016; 16,48}

\usepackage{amsmath,amsthm,amscd,amsxtra,amssymb,
}
\usepackage{graphicx}
\usepackage[all]{xy}

\newtheorem{theo}{Theorem}[section]
\newtheorem{pro}[theo]{Proposition}
\newtheorem{prodef}[theo]{Proposition and Definition}
\newtheorem{lem}[theo]{Lemma}
\newtheorem{cor}[theo]{Corollary}

\theoremstyle{remark}
\newtheorem{re}[theo]{Remark}

\newcommand{\putv}[5]{\put(#1,#2){\vector(#3,#4){#5}}}

\newcommand{\putbb}[3]{\put(#1,#2){\makebox(0,0)[b]{#3}}}

\newcommand{\CH}{\mathcal{H}}

\newcommand{\CV}{\mathcal{V}}

\newcommand{\BL}{\mathbb L}

\newcommand{\BN}{\mathbb N}

\newcommand{\BZ}{\mathbb Z}

\newcommand{\ZA}{\mathbf{A}}
\newcommand{\ZB}{\mathbf{B}}

\newcommand{\ZE}{\mathbf{E}}
\newcommand{\ZF}{\mathbf{F}}

\newcommand{\ZK}{\mathbf{K}}

\newcommand{\ZP}{\mathbf{P}}

\newcommand{\ZS}{\mathbf{S}}

\newcommand{\zi}{\mathbf{i}}

\newcommand{\zm}{\mathbf{m}}

\newcommand{\zo}{\mathbf{o}}
\newcommand{\zp}{\mathbf{p}}

\newcommand{\zr}{\mathbf{r}}
\newcommand{\zs}{\mathbf{s}}

\newcommand{\zu}{\mathbf{u}}

\renewcommand{\a}{\alpha}
\renewcommand{\b}{\beta}
\newcommand{\g}{\gamma}

\renewcommand{\l}{\lambda}

\newcommand{\f}{\varphi}

\newcommand{\G}{\varGamma}
\newcommand{\D}{\varDelta}

\renewcommand{\L}{\varLambda}
\renewcommand{\Xi}{\varXi}
\renewcommand{\Pi}{\varPi}
\renewcommand{\Sigma}{\varSigma}

\newcommand{\F}{\varPhi}
\renewcommand{\P}{\varPsi}

\newcommand{\nin}{\not\in} 
\newcommand{\sbs}{\subset} 
\newcommand{\psbs}{\subsetneqq} 
\newcommand{\nsbs}{\not\subset} 
\newcommand{\sps}{\supset} 
\newcommand{\is}{\simeq} 
\newcommand{\subis}{\lesssim} 
\renewcommand{\le}{\leqslant} 
\renewcommand{\ge}{\geqslant} 

\DeclareMathOperator{\Ker}{Ker}
\DeclareMathOperator{\Coker}{Coker}
\DeclareMathOperator{\Imm}{Im}

\DeclareMathOperator{\End}{End}

\DeclareMathOperator{\Supp}{Supp}

\DeclareMathOperator{\Soc}{Soc}
\DeclareMathOperator{\Tr}{Tr}
\DeclareMathOperator{\LL}{L}

\newcommand{\defug}{:\,=}

\newcommand{\FP}{\ZF\ZP}
\newcommand{\ESFP}{\ZE\ZS\ZF\ZP}

\newcommand{\KO}{\ZK_0}

\newcommand{\Simp}{{\ZS\zi\zm\zp}}
\newcommand{\Prosimp}{{\ZP\zr\zo\zs\zi\zm\zp}}
\newcommand{\Prim}{{\ZP\zr\zi\zm}}

\newcommand{\map}[3]{#1\colon #2\to#3}

\newcommand{\lmap}[3]{#1\colon#2\longrightarrow#3}

\numberwithin{equation}{section}

\begin{document}

\author{Giuseppe Baccella and Leonardo Spinosa}

\address{Dipartimento di Ingegneria e scienze dell'informazione e matematica\newline
Universit\`a di L'Aquila, Via Vetoio, 67100 L'Aquila, Italy}
\email{baccella@univaq.it}

\address{University of Ferrara, Department of Mathematics and Computer Science\newline
Via Machiavelli 30, Ferrara, I-44121, Italy}
\email{leonardo.spinosa@unife.it}

\begin{abstract}
        If $R$ is a regular and semiartinian ring, it is proved that the following conditions are equivalent: (1) $R$ is unit-regular, (2) every factor ring of $R$ is directly finite, (3) the abelian group $\KO(R)$ is free and admits a basis which is in a canonical one to one correspondence with a set of representatives of simple right $R$-modules. For the class of semiartinian and unit-regular rings the canonical partial order of $\KO(R)$ is investigated. Starting from any partially ordered set $I$, a special dimension group $G(I)$ is built and a large class of semiartinian and unit-regular rings is shown to have the corresponding $\KO(R)$ order isomorphic to $G(\Prim_{R})$, where $\Prim_{R}$ is the primitive spectrum of $R$. Conversely, if $I$ is an artinian partially ordered set having a finite cofinal subset, it is proved that the dimension group $G(I)$ is realizable  as $\KO(R)$ for a suitable semiartinian and unit-regular ring $R$.
\end{abstract}

\maketitle

\setcounter{section}{-1}
\section{Introduction}

The class of semiartinian, Von Neumann regular rings (some authors prefer to speak of Loewy rings, instead of semiartinian rings) turns out to be a very large one. As it was remarked in \cite{Bac:12} more than twenty years ago, it includes several of the most important examples of rings that, over the years, were found in order to detect, display and understand uneven and ``pathological'' behaviors in the theory of regular rings. More recently, two new constructions have definitely shown that the class of semiartinian and regular rings is large and complex enough to deserve more interest and investigations. Concerning the first one, we first recall that one basic property of a semiartinian and regular ring $R$ (firstly remarked by Camillo and Fuller in \cite{CamilloFuller:2}) is that its primitive spectrum $\Prim_R$ (namely the set of primitive ideals ordered by inclusion) is artinian, i. e. satisfies the minimum condition; often it has a finite cofinal subset too. Conversely, it has been shown in \cite{Bac:16} that, given an artinian partially ordered set I and a field $D$, there exists a semiartinian and unit-regular $D$-algebra $D_I$, together with an injective and order preserving map from $I$ to $\Prim_R$; that map is an isomorphism if and only if $I$ has a finite cofinal subset, otherwise the image of this map leaves apart exactly one primitive ideal. In addition $D_I$ turns out to be a right $V$-ring, and it is also a left $V$-ring if and only if $I$ is an antichain. The second construction was given by Abrams, Rangaswamy and M. Siles Molina in \cite{AbrRangSMol:1}, where they characterized those graphs $E$ whose Leavitt path algebra $L_K(E)$ is semiartinian, in which the case it is unit-regular as well; they showed that, given any ordinal $\xi$ and a field $K$, there exists a graph $E$ such that $L_K(E)$ is semiartinian with Loewy length $\xi+1$.

In the present work we resume the study of the Grothendieck group $\KO(R)$ of a semiartinian and regular ring $R$, which was started in \cite{BacCiamp:14}, where it was shown that when $R$ is in addition unit-regular, then $\KO(R)$ is free with a basis which is in one to one correspondence with a set of representatives of simple right $R$-modules. With the main result of the first section we invert the above statement. To be more precise, assume that $R$ is a semiartinian and regular ring, with Loewy length $\xi+1$ and Loewy chain of ideals $(L_{\a})_{\a\le\xi+1}$. If $A$ is a finitely generated projective right $R$-module, then there is a smallest nonzero (successor) ordinal $h(A)\le\xi+1$ such that $A = AL_{h(A)}$; say that $A$ is \emph{eventually simple} in case $A/AL_{h(A)-1}$ is simple. Then there exists a \emph{complete and irredundant} set $\ZA$ of eventually simple finitely generated projective right $R$-modules, in the sense that for every simple right $R$-module $U$ there is a unique $A\in\ZA$ such that $U\is A/AL_{h(A)-1}$. By denoting with $[A]$ the stable isomorphism class of $A$, it was shown in \cite{BacCiamp:14} that the set $[\ZA] = \{[A] \mid A\in\ZA\}$ always generates the group $\KO(R)$ and, if $R$ is unit-regular, then $\KO(R)$ is free with  $[\ZA]$ as a basis; in the latter case, let us say that $[\ZA]$ is a \emph{simp-basis} for $\KO(R)$. The Main Theorem (Theorem \ref{theo main}) of the first section of the present work states that the following three conditions are equivalent for a semiartinian and regular ring $R$ (Theorem \ref{theo main}): (1) $R$ is unit-regular, (2) every factor ring of $R$ is directly finite, (3) $\KO(R)$ is free and admits a simp-basis. Thus our result gives also a contribution to the solution of Problem 3 of Goodearl book \cite{Good:3}, where one asks if, given a regular ring $R$, the condition of direct finiteness for all factor rings of $R$ entails unit-regularity of $R$: the answer is ``yes'' when $R$ is semiartinian.

In the second section we investigate about the possible relationship between the partial order (by inclusion) of the primitive spectrum $\Prim_R$ of a semiartinian and regular ring $R$, the algebraic pre-order of the additive monoid $\CV(R)$ of isomorphism classes of finitely generated and projective right $R$-modules, and the canonical pre-order of the group $\KO(R)$. The starting point is the observation that if $\Simp_R$ is a set of representatives of simple right $R$-modules, then the assignment $U\mapsto r_R(U)$ defines a bijection from $\Simp_R$ to $\Prim_R$; consequently we may define a partial order $\preccurlyeq$ in $\Simp_R$ by declaring that $U\preccurlyeq V$ if and only if $r_R(U)\sbs r_R(V)$ (the first author had defined in \cite{Bac:15} a natural partial order in $\Simp_R$ for a general semiartinian ring $R$; in case $R$ is regular, that order becomes the one we have just defined). If $A$ and $B$ are members of a complete and irredundant set $\ZA$ of eventually simple and finitely generated and projective right $R$-modules, by considering the simple modules $U= A/AL_{h(A)-1}$ and $V= B/BL_{h(B)-1}$ it is true that $A\subis B$ implies $U\preccurlyeq V$ but, in general, the converse may fail. We say that the set $\ZA$ is \emph{order representative} in case the converse holds. If such a set $\ZA$ exists, then the set $\{\overline{A}\mid A\in\ZA\}$ is a generating subset of $\CV(R)$ which is order isomorphic to $\Prim_R$. The main result of this section concerns the obvious question: which semiartinian and regular rings admit such an order representative set? First note that if $H$ is an ideal of $R$, then $\Simp_{R/H}$ is an upper subset of $\Simp_R$ and the assignment $H\mapsto \Simp_{R/H}$ defines an injective and order-reversing map $\F$ from the lattice $\BL_{2}(R)$ of ideals of $R$ to the lattice $\Uparrow
\!\!\Simp_{R}$ of upper subsets of $\Simp_R$, this latter endowed with the above partial order. The map $\F$ need not be onto; the ring $R$ was said in \cite{Bac:16} to be \emph{very well behaved} in case $\F$ is onto. Here we show that $R$ admits an order representative set of eventually simple and finitely generated and projective right $R$-modules if and only if $R$ is very well behaved.

In the third section, which concludes our work, we analize the order structure of $\KO(R)$, when $R$ is a semiartinian, unit-regular and very well behaved ring. First, to every partially ordered set $I$ we associate the partially ordered and directed abelian group $G(I) = (\BZ^{(I)}, M(I))$, where $\BZ^{(I)}$ is the usual free abelian group with basis $I$ and $M(I)$ is the submonoid whose nonzero elements are all (finite) linear combinations $a_{1}i_{1}+\cdots+a_{r}i_{r}$ of elements of $I$, with coefficients in $\BZ$, where $a_{\a}>0$ in case $i_{\a}$ is maximal in $\{i_{1},\ldots,i_{r}\}$. Then it turns out that $G(I)$ is a dimension group (i. e. it is directed, unperforated and satisfies the Riesz decomposition property), which has an order unit if and only if $I$ has a finite cofinal subset $F$, in which the case the sum $u$ (in $G(I)$) of all elements of $F$ is an order unit for $G(I)$. Then we prove that there is a natural order isomorphism from $G(\Prim_R)$ to $\KO(R)$, which restricts to an isomorphism from $M(\Prim_R)$ to $\CV(R)$. As outlined at the beginning of this introduction, every artinian, partially ordered set $I$ with a finite cofinal subset is realizable as $\Prim_R$ for some semiartinian, unit-regular and very well behaved ring $R$. Thus our result gives a partial answer to the Problem 29 in \cite{Good:3}, where it is asked which directed abelian groups with order-unit are isomorphic to $(\KO(R),[R])$ for some regular (or unit-regular) ring $R$. The same result gives also a partial solution to the \emph{Realization Problem for Von Neumann regular rings}, which asks what are the conical refinement monoids which are isomorphic to $\CV(R)$ for some regular ring $R$ (see \cite{Ara:011}, \cite{AraGood:10}): those of the form $M(I)$, where $I$ is an artinian partially ordered set with a finite cofinal subset, are among them and they are tame in the sense of \cite{AraGood:10}).

Concerning notations, we follow the (by now) universally established practice in rings and modules theory. However, we use the symbol $A\sbs B$ with the meaning ``$A$ is a subset of $B$'' and write $A\psbs B$ to indicate that ``$A$ is a proper subset of $B$''. Finally, we always consider the number $0$ as a member of the set $\BN$ of natural numbers.

\section{Direct finiteness versus unit-regularity.}

Given a ring $R$, we denote with $\FP(R_R)$ the class of finitely generated and projective right $R$-modules and, for every $A\in\FP(R_R)$, we denote with $\overline{A}$ the class of all elements of $\FP(R_R)$ which are isomorphic to $A$. Then the \emph{set} $\CV(R)\defug\{\overline{A}\mid A\in\FP(R_R)\}$ has a canonical structure of an abelian monoid given by the addition defined by the rule
\[
\overline{A} +\overline{B} \defug \overline{A\oplus B}.
\]
The \emph{algebraic preorder} in $\CV(R)$ is the reflexive and transitive relation $\leqslant$ which prescribes that, given $\overline{A},\overline{B}\in\CV(R)$, one has $\overline{A}\leqslant\overline{B}$ if and only if $B\is A\oplus C$ for some $C\in\FP(R_R)$. Two modules $A,B\in\FP(R_R)$ are said to be \textbf{stably isomorphic} (write $[A] = [B]$) if there is some positive integer $n$ such that $A\oplus R^n \is B\oplus R^n$. Stable isomorphism is obviously an equivalence relation in the class $\FP(R_R)$ and, by denoting with $[A]$ the equivalence class of a given $A\in\FP(R_R)$, the factor class $S = \{[A]\mid A\in\FP(R_R)\}$ is a set, given that $\overline{A}\sbs[A]$. $S$ is an abelian monoid with respect to the natural operation $[A]+[B] \defug [A\oplus B]$ and every element of $S$ is cancellable. The Grothendieck group $\KO(R)$ of $R$ can be defined as the Grothendieck group of the monoid $S$; formally:
\[
\KO(R) = \{[A]-[B]\mid A,B\in\FP(R_R)\},
\]
where $[A]-[B] = [C]-[D]$ if and only if $[A]+[D] = [B]+[C]$. If $\ZA$ is a subset of $\FP_{R}$, we shall denote by $[\ZA]$ the subset $\{[A]\mid A\in\ZA\}$ of $\KO(R)$.

If $R$ is a unit-regular ring, since all finitely generated projective right $R$-modules are cancellable from direct sums, we have that $\overline{A} = [A]$ for every $A\in\FP(R_R)$, so that $\CV(R)$ coincides with the above monoid $S$ and the algebraic preorder is a partial order.

Unless otherwise specified, in what follows $R$ denotes a regular and semiartinian ring with Loewy length $\xi+1$ and we set $L_{\a} \defug \Soc_{\a}(R_R)$ for every ordinal $\a$ (we shall keep consistently this notation throughout the present paper). If $M$ is a right $R$-module, we define the ordinal $h(M) = \min\{\,\a\leqslant \LL(R_R)\mid ML_{\a} = M\}$; clearly $h(M)$ is a successor ordinal if $M$ is finitely generated. If $x\in R$, then it is easy to see that
\[
h(xR) = \min\{\,\a\leqslant\LL(R_R)\mid x\in L_{\a}\};
\]
we write $h(x)$ for $h(xR)$. If $A$ is a finitely generated projective $R$-module, then for every ordinal $\a$ the equality
\[
AL_{\a} = \Soc_{\a}(A)
\]
holds (see \cite[Proposition 1.5]{Bac:15}), so that $h(A)$ is precisely the Loewy length of $A$.

If $\ZA$ is any class of right $R$-modules, for every ordinal $\a$ we consider the subclass
\[
\ZA_{\a} \defug \{A\in\ZA\mid h(A) = \a+1\}.
\]
As it was observed in \cite{BacCiamp:14}, for every ordinal $\a\le\xi$ we can choose a set $E_{\a}$ of idempotents of $L_{\a+1}\setminus L_{\a}$ in such a way that $\{eR/eL_{\a}\mid e\in E_{\a}\}$ is an irredundant set of representatives of all simple and projective right $R/L_{a}$-modules and
\[
\bigcup_{\a\le\xi}\{eR/eL_{\a}\mid e\in E_{\a}\} = \Simp_{R}.
\]

If $A$ is a finitely generated and projective right $R$-module, then $A/AL_{h(A)-1}$ is semisimple; let us say that $A$ is \emph{eventually simple} in case $A/AL_{h(A)-1}$ is simple and let us define class
\[
\ESFP_{R} \defug \{A\in\FP_{R}\mid\text{$A$ is eventually simple}\}.
\]
We say that a subset $\ZA\sbs\ESFP_{R}$ is \emph{complete and irredundant} if for every simple right $R$-module $U$ there is a \underbar{unique} $A\in\ZA$ such that $A/AL_{h(A)-1}\is U$; in particular the set $\bigcup_{\a\le\xi}\{eR\mid e\in E_{\a}\}$ is a complete and irredundant subset of $\ESFP_{R}$.

\begin{prodef}\label{pro simpgenKO}
Let $R$ be a regular and semiartinian ring with Loewy length $\xi+1$ and let $\ZA$ be a complete and irredundant subset of $\ESFP_{R}$. Then:
\begin{enumerate}
  \item For every $\a\le\xi+1$ the equality
\[
(*)_{\a}\hspace{25mm} \Ker(\KO(\f_{\a})) = \langle[A]\mid A\in\ZA\text{ and }A = AL_{\a}\rangle \hspace{27mm}
\]
holds, where $\map{\f_{\a}}{R}{R/L_{\a}}$ is the canonical projection. Consequently $[\ZA]$ genera\-tes $\KO(R)$ and we say that $[\ZA]$ is a \emph{simp-generating} subset of $\KO(R)$.
  \item If $[\ZA]$ is a basis for $\KO(R)$, then we say that $[\ZA]$ is a \emph{simp-basis} for $\KO(R)$. If it is the case, then for every $\a\le\xi+1$ the set
\begin{equation}\label{eq factorbasis}
\{[A/AL_{\a}]\mid A\in\ZA \text{ and $h(A)\ge\a+1$}\}
\end{equation}
is a simp-basis for $\KO(R/L_{\a})$.
  \item If $\KO(R)$ has a simp-basis, then $[\ZB]$ is a simp-basis of $\KO(R)$ whenever $\ZB$ is a complete and irredundant subset of $\ESFP_{R}$.
\end{enumerate}

\end{prodef}
\begin{proof}

(1) It follows from the assumptions that
\[
\{A/AL_{\a}\mid A\in\ZA_{\a}\} = \Prosimp_{R/L_{\a}}.
\]
Moreover $\{[\ZA_{\a}]\mid \a\le\xi\}$ is a partition of $[\ZA]$. For every ordinal $\a$ let us consider the subgroup
\[
G_{\a}  = \langle[A]\mid A\in\ZA\text{ and }A = AL_{\a}\rangle
\]
of $\KO(R)$. According to \cite[Proposition 15.15]{Good:3}, for every $\a\le\xi+1$ we have that
\begin{equation}\label{eq KO}
\Ker(\KO(\f_{\a})) = \langle [P]\mid \text{$P\in \FP_{R}$ and $P = PL_{\a}$}\rangle,
\end{equation}
consequently the inclusion $G_{\a} \sbs \Ker(\KO(\f_{\a}))$ is clear.
Now $(*)_{0}$ is obviously true. Given an ordinal $\a>0$, assume that $(*)_{\b}$ is true whenever $\b<\a$. If $\a$ is a limit ordinal, then $L_{\a} = \bigcup_{\b<\a}L_{\b}$ and we have the following equalities:
\begin{align*}
\Ker(\KO(\f_{\a})) &= \bigcup_{\b<\a}\langle [P]\in [\FP_{R}]\mid P = PL_{\b}\rangle = \bigcup_{\b<\a}\Ker(\KO(\f_{\b})) \\
                   &= \bigcup_{\b<\a}\langle[A]\mid A\in\ZA\text{ and }A = AL_{\b}\rangle= \langle[A]\mid A\in\ZA\text{ and }A= AL_{\a}\rangle.
\end{align*}
Assume that $\a = \b+1$ for some $\b$ and take any idempotent $e\in L_{\a}\setminus L_{\b}$. If we show that $[eR]\in G_{\a}$, then the inclusion $\Ker(\KO(\f_{\a})) \sbs G_{\a}$ will follow again from \cite[Proposition 15.15]{Good:3}. Since $(e+L_{\b})(R/L_{\b})$ is a semisimple right ideal of $R/L_{\b}$, then there are $A_{1},\ldots,A_{n}\in\ZA_{\b}$ and positive integers $k_{1},\ldots,k_{n}$ such that
\[
(e+L_{\b})(R/L_{\b}) \is (A_{1}/A_{1}L_{\b})^{k_{1}} \oplus\cdots\oplus (A_{n}/A_{n}L_{\b})^{k_{n}}.
\]
Consequently
\[
[eR]-k_{1}[A_{1}]-\cdots-k_{n}[A_{n}]\in\Ker(\KO(\f_{\b}))
\]
and, since $k_{1}[A_{1}]+\cdots+k_{n}[A_{n}]\in\langle[\ZA_{\b}]\rangle$, it follows from the inductive hypothesis that
\[
[eR]\in G_{\b}+\langle[\ZA_{\b}]\rangle = G_{\a},
\]
as wanted.

(2) Assume that $[\ZA]$ is a simp-basis for $\KO(R)$. Given $\a\le\xi$, it follows from the equality $(*)_{\a}$ that the set $\bigcup_{\b<\a}[\ZA_{\b}]$ is a basis for  $\Ker(\KO(\f_{\a}))$; consequently $\KO(R/L_{\a})\is\KO(R)/\Ker(\KO(\f_{\a}))$ is free with the set \eqref{eq factorbasis} as a basis.

(3) Given two complete and irredundant subsets $\ZA\sbs\ESFP_{R}$ and $\ZB\sbs\ESFP_{R}$, assume that $[\ZA]$ is a simp-basis of $\KO(R)$. We will show that $[\ZB]$ too is a simp-basis by pro\-ving that $\bigcup_{\g<\a}[\ZB_{\g}]$ is independent for every ordinal $\a\ge 1$. Since $\ZB_{0}$ and $\ZA_{0}$ are both irredundant representative sets of simple and projective right $R$-modules, then $[\ZB_{0}] = [\ZA_{0}]$ and hence $[\ZB_{0}]$ is independent; this proves that our statement is true when $\a = 1$. Given an ordinal $\a$ with $0<\a\le\xi+1$, assume that $\bigcup_{\g<\b}[\ZB_{\g}]$ is independent for every $\b<\a$. If $\a$ is a limit ordinal, then $\bigcup_{\g<\a}[\ZB_{\g}]$ is independent. Suppose that $\a = \b+1$ for some $\b$, take distinct elements $[B_{1}],\ldots,[B_{n}]\in \bigcup_{\g<\a}[\ZB_{\g}]$ and assume that there are integers $k_{1},\ldots,k_{n}$ such that $k_{1}[B_{1}]+\cdots+k_{n}[B_{n}] = 0$. We may assume that there is an integer $r$ with $1\le r\le n$ such that $h(B_{i})\le \b$ if $1\le i\le r$, while $h(B_{i}) = \a$ if $r<i\le n$. Then
\begin{equation}\label{eq simpbasis2}
k_{r+1}[B_{r+1}]+\cdots+k_{n}[B_{n}] = -k_{1}[B_{1}]-\cdots-k_{r}[B_{r}]
\end{equation}
and, since the second member belongs to $\Ker(\KO(\f_{\b}))$ by
\eqref{eq KO}, we infer that
\[
k_{r+1}[B_{r+1}/B_{r+1}L_{\b}]+\cdots+k_{n}[B_{n}/B_{n}L_{\b}] = 0
\]
in $\KO(R/L_{\b})$. Now $B_{r+1}/B_{r+1}L_{\b},\ldots,B_{n}/B_{n}L_{\b}$ are pairwise non isomorphic simple and projective right $R/L_{\b}$-modules, therefore for each $i$ with $r<i\le n$ there is $A_{i}\in\ZA_{\b}$ such that $B_{i}/B_{i}L_{\b} \is A_{i}/A_{i}L_{\b}$ and so $[B_{i}/B_{i}L_{\b}] = [A_{i}/A_{i}L_{\b}]$. Thus it follows from the above proved property (2) that $k_{r+1} = \cdots = k_{n} = 0$. Since $[B_{1}],\ldots,[B_{r}]$ are independent by the inductive hypothesis, we conclude that $k_{1} = \cdots = k_{r} = 0$ as well and this shows that $\bigcup_{\b<\a}[\ZB_{\b}]$ is independent.
\end{proof}

We know from \cite[Theorem 1]{BacCiamp:14} that if $R$ is unit-regular, then $\KO(R)$ admits $\bigcup_{\a\le\xi}\{[eR]\mid e\in E_{\a}\}$ as a simp-basis. Our main objective in this section will be to show that a regular and semiartinian ring $R$ is unit-regular if and only if the group $\KO(R)$ is free with a simp-basis, if and only if all factor rings of $R$ are directly finite. We observe that, for a regular and semiartinian ring $R$, the group $\KO(R)$ may be free without admitting a simp-basis; in fact, let $R$ be the regular ring of \cite[Example 1]{MenalMoncasi:1}. Then $\KO(R) \is \BZ$; however $R$ is primitive and semiartinian with Loewy length two and has two isoclasses of simple right $R$-modules, therefore $\KO(R)$ cannot have a simp-basis. Here the point is that $R$ is not unit-regular and, more, it is not directly finite either.

In order to reach our goal we need some preparation.

Assume that $R$ is a prime ring with $S = \Soc(R)\ne 0$ and let $U$ be a minimal left ideal of $R$. By considering $U$ as a right vector space over the division ring $D = \End({_RU})$ and, by setting $Q = \End(U_D)$ and $K = \Soc(Q)$, it is well known that $R$ can be viewed as a dense subring of $Q$ and $S = R\cap K$. If $e$ is any primitive idempotent of $R$, namely $Re\is U$, then $Re = Qe$. As a result $S = QS$, while $S = SQ$ if and only if $S = K$; moreover, if $e,f$ are \emph{any} idempotents of $R$ with $f\in S$, then $eRf = eQf$. In the following lemmas we keep these settings and notations.

Recall that if $a, a'$ are elements of a ring $R$, then $a'$ is a \emph{quasi-inverse} for $a$ when $a = aa'a$. If it is the case, then also $b = a'aa'$ is a quasi-inverse for $a$ and $a$ is a quasi-inverse for $b$.

\begin{lem}\label{invertmodL 1}
Let $R$ be a prime ring with $S = \Soc(R)\ne 0$ and let $a,b\in R$ be such that $a = aba$ and $b = bab$. Then the following conditions are equivalent, where the dimensions are relative to the division ring $D$:
\begin{enumerate}
    \item $\dim(\Ker(a))<\infty$ (resp. $\dim(\Coker(a))<\infty$).
    \item $a+S$ is left (resp. right) invertible in $R/S$.
    \item $1-ba\in S$ (resp. $1-ab\in S$).
\end{enumerate}
\end{lem}
\begin{proof}
Let us consider the idempotents $e = ab$ and $f = ba$. Then $a = ea = af$, $b = fb = be$ and we have the equalities
\begin{gather}
aR = eR, \quad Ra = Rf, \label{eq aR = eR Ra = Rf}\\
bR = fR, \quad Rb = Re; \label{eq bR = fR Rb = Re}
\end{gather}
it follows that, as $D$-subspaces of $U$,
   \begin{gather}
\Ker(a) = (1-f)U \is \Coker(b),  \label{eq kercoker aba bab}\\
\Ker(b) = (1-e)U \is \Coker(a).\label{eq kercoker bab aba}
\end{gather}

(1)$\Rightarrow$(3) If $\Ker(a) = (1-f)U$ is finite dimensional, then
\[
1-ba = 1-f\in R\cap K = S.
\]
Similarly, if $\Coker(a) = (1-e)U$ is finite dimensional, then
\[
1-ab = 1-e\in R\cap K = S.
\]

(3)$\Rightarrow$(2) is obvious.

(2)$\Rightarrow$(1) Given any element $c\in R$, we have that $\Ker(a)\sbs \Imm(1-ca)$. If $1-ca\in S$, that is, $\Imm(1-ca)$ is finite dimensional, then so is $\Ker(a)$. On the other hand, if $1-ac\in S$, then $\Ker(1-ac)$ has finite codimension; since $\Ker(1-ac)\sbs \Imm(a)$, it follows that $\Imm(a)$ has finite codimension.
\end{proof}

\begin{lem}\label{invertmodL 2}
Let $R$ be a prime and regular ring with $S = \Soc(R)\ne 0$ and let $a\in R$. Then the following conditions are equivalent:
\begin{enumerate}
    \item $\dim(\Ker(a))<\infty$ and  $\dim(\Ker(a))\le\dim(\Coker(a))$.
    \item There is a left invertible element $a'\in R$ such that $a-a'\in S$.
\end{enumerate}
Similarly, the following conditions are equivalent:
\begin{enumerate}
\setcounter{enumi}{2}
    \item $\dim(\Coker(a))<\infty$ and  $\dim(\Coker(a))\le\dim(\Ker(a))$.
    \item There is a right invertible element $a'\in R$ such that $a-a'\in S$.
\end{enumerate}
Consequently $a$ is congruent modulo $S$ to a unit of $R$ if and only if
\begin{equation}\label{eq congrunit}
\dim(\Ker(a)) = \dim(\Coker(a))<\infty.
\end{equation}
\end{lem}
\begin{proof}
Let $b$, $e$ and $f$ have the same meaning as in Lemma \ref{invertmodL 1} and its proof.

(1)$\Rightarrow$(2) Assuming (1), then it follows from Lemma \ref{invertmodL 1} that $1-f\in S$. There exists $c\in Q$ such that $cfU = 0$ and the restriction of $c$ to $(1-f)U = \Ker(a)$ is a monomorphism into $(1-e)U = \Coker(a)$. As a result, since $QS = S$, we have that
\[
c\in (1-e)Q(1-f) = (1-e)R(1-f) \sbs S.
\]
By taking $a' = a+c\in R$, it is easy to check that $\Ker(a') = 0$ . Then $a'$ has a left inverse in $Q$. Inasmuch as $R$ is a regular ring, we infer that $a'$ has a left inverse in $R$ too and $a-a'\in S$, as wanted.

(2)$\Rightarrow$(1) Assume (2) and set $c = a-a'$. Then it follows that $\Ker(a)\cap\Ker(c) = 0$ and therefore $\dim(\Ker(a))<\infty$. There are some subspaces $A$, $B$, $C$ of $U_D$ such that
\[
U = \Ker(a)\oplus\Ker(c)\oplus A = aU\oplus B = a'U\oplus C.
\]
As a result we obtain:
\begin{align*}
aA\oplus a\Ker(c)\oplus B &= a'A\oplus a'\Ker(a)\oplus a'\Ker(c)\oplus C \\
                          &= a'A\oplus c\Ker(a)\oplus a\Ker(c)\oplus C,
\end{align*}
therefore
\[
aA\oplus B \is a'A\oplus c\Ker(a)\oplus C.
\]
Noting that $\dim(aA) = \dim(A) = \dim(a'A)$ is a finite cardinal, we conclude that
\begin{align*}
\dim(\Coker(a)) &= \dim(B) = \dim(c\Ker(a))+\dim(C) \\
&= \dim(\Ker(a))+\dim(\Coker(a'))
\end{align*}
and hence (1) holds.

If $a'$ is a unit, then $\dim(\Coker(a')) = 0$ and the above equalities show that \ref{eq congrunit} holds.

(3)$\Rightarrow$(4) Assume (3) and note that $1-e = 1-ab\in S$ by Lemma \ref{invertmodL 1}. There is $c\in Q$ such that $ceU = 0$ and the restriction of $c$ to $\Coker(a) = (1-e)U$ is a monomorphism into $\Ker(a) = (1-f)U$, so that
\[
c\in (1-f)Q(1-e) = (1-f)R(1-e) \sbs S
\]
(remember that $QS = S$). In particular $cR\sbs (1-f)R$. By the regularity of $R$, there are two orthogonal idempotents $g,h\in R$ such that $1-f = g+h$ and $cR = gR$. Consequently $(1-f)U = gU\oplus hU$ and $c\in gR(1-e)$, therefore $c$ induces an isomorphism from $(1-e)U$ to $cU = gU$. Since $g\in S$, it follows that $(1-e)Qg = (1-e)Rg$, thus there exists $d\in (1-e)Rg$ such that
\[
cd = 1-e \quad \text{and} \quad dc = g.
\]
We claim that $a' = a+d$ is right invertible in $R$. Given $u\in U$, let $v$ be the unique element of $gU$ such that $dv = (1-e)u$. Noting that $g = g(1-f) = (1-f)g$, we obtain:
\begin{align*}
a'(bu+v) &= (a+d)(bu+v) = abu + av + dbu + dv \\
            &= eu + af(1-f)gv + dg(1-f)fbu + (1-e)u \\
            &= eu + (1-e)u \\
            &= u.
            \end{align*}
This proves that $a'U = U$, hence $a'$ is right invertible in $Q$ and so is right regular. Again, the regularity of $R$ implies that $a'$ is right invertible in $R$.

(4)$\Rightarrow$(3) Assume (4) and set $c = a-a'$. Then $U = aU + cU$ and therefore $\dim(\Coker(a))<\infty$. Let $b\in R$ be a right inverse for $a'$. Then we have that
\[
ab - cb = a'b = 1
\]
and $cb\in S$, meaning that $ab$ is congruent modulo $S$ to a unit. As we have shown previously, this implies that
\[
\dim(\Ker(ab)) = \dim(\Coker(ab)).
\]
By using the fact that $b$ is injective and the equivalence (1)$\Leftrightarrow$(2) we conclude that
\[
\dim(\Ker(a)) = \dim(\Ker(ab)) = \dim(\Coker(ab)) \ge \dim(\Coker(a)).
\]
It remains to show that if \ref{eq congrunit} holds, then $a$ is congruent modulo $S$ to a unit. So, assume that $a$ satisfies \ref{eq congrunit}. Then (1) holds and so there is a left invertible element $a'\in R$ such that $a-a'\in S$; moreover, as we have seen in the prof of the implication (2)$\Rightarrow$(1), we have the equality
\[
\dim(\Coker(a)) = \dim(\Ker(a))+\dim(\Coker(a')).
\]
The assumption implies now that $\dim(\Coker(a')) = 0$, so $a'$ is a unit of $Q$ and hence $a'$ is a unit of $R$ too, because of the regularity of $R$.
\end{proof}

\begin{lem}\label{lem directfin}
Assume that $R$ is a prime and regular ring with a nonzero socle $S$. Then $R$ is directly finite if and only if the following conditions hold:
\begin{enumerate}
    \item $R/S$ is directly finite;
    \item Every unit of $R/S$ lifts to a unit of $R$.
\end{enumerate}
\end{lem}
\begin{proof}
Suppose that (1) and (2) hold and let $a,b\in R$ be such that $ab = 1$. Then $\bar{a}\bar{b} = \bar{1} = \bar{b}\bar{a}$ in $R/S$ and so $\bar{a} = \bar{c}$ for some unit $c\in R$. It follows from Lemma \ref{invertmodL 2} that $\dim(\Ker(a)) = \dim(\Coker(a))$ and, since $\Coker(a) = 0$, then $\Ker(a) = 0$ as well. Thus $a$ is invertible in $Q$ and the regularity of $R$ implies that $a$ has its inverse in $R$.

Conversely, assume that $R$ is directly finite, let $x$ be a right invertible element of $R/S$ and choose $a\in R$ such that $x = \bar{a}$. Then $\Coker(a)$ is finite dimensional by Lemma \ref{invertmodL 1}. If $\dim(\Ker(a)) \le \dim(\Coker(a))$ (resp. $\dim(\Ker(a)) \ge \dim(\Coker(a))$), then it follows from Lemma \ref{invertmodL 2} that there is some left (resp. right) invertible element $a'\in R$ such that $a-a'\in S$. Then $a'$ is invertible by the direct finiteness of $R$ and, consequently, $x = \bar{a} = \bar{a'}$ is invertible. We conclude that $R/S$ is directly finite and the above argument shows that each unit of $R/S$ lifts to a unit of $R$.
\end{proof}

In what follows $R$ denotes a regular ring with essential socle $S$. Let $(U_{\l})_{\l\in\L}$ be a representative family of minimal left ideals, where $U_{\l}\ncong U_{\mu}$ if $\l\ne\mu$. Thus
\[
_{R}\Prosimp = \{U_{\l}\mid\l\in\L\};
\]
moreover, given $\l\in\L$, the $U_{\l}$-homogeneous component of $S$ is the ideal $S_{\l}\defug U_{\l}R$ and hence
\[
S = \bigoplus_{\l\in\L}S_{\l}.
\]
For every $\l\in\L$ we have that $U_{\l}$ is a right vector space over the division ring $D_{\l}\defug \End({_R(U_{\l})})$. We set $Q_{\l} \defug \End((U_{\l})_{D_{\l}})$ and $K_{\l}\defug l_{R}(U_{\l})$. Thus the ring $R_{\l} \defug R/K_{\l}$ imbeds canonically as a dense subring into $Q_{\l}$ and $S_{\l}$ is identified with $\Soc(R_{\l})$ via the canonical projection $\map{\pi_{\l}}{R}{R_{\l}}$. Since $\bigcap_{\l\in\L}K_{\l} = 0$, we have the imbeddings of rings
\[
R\hookrightarrow \prod_{\l\in\L}R_{\l}\hookrightarrow \prod_{\l\in\L}Q_{\l}.
\]
If $a\in R$, for every $\l\in\L$ we denote by $a_{\l}$ the element $\pi_{\l}(a)$ and we identify $a$ with the element $(a_{\l})_{\l\in\L}\in \prod_{\l\in\L}R_{\l}$. In the proof of the following lemma we keep the above notations and settings.

\begin{lem}\label{lem directfin1}
Assume that $R$ is a regular ring with essential socle $S$ and suppose that all primitive factor rings of $R$ are directly finite. Then every unit of $R/S$ lifts to a unit of $R$.
\end{lem}
\begin{proof}
Given $a\in R$, assume that $a+S$ is a unit of $R/S$. Then there are finite subset $\L',\L''\sbs\L$ such that $a_{\l}$ is a unit of $R_{\l}$ if $\l\in\L\setminus(\L'\cup\L'')$, while if $\l\in\L'$ (resp. $\l\in\L''$), then  $a_{\l}+S_{\l}$ is right (resp. left) invertible in $R_{\l}/S_{\l}$. We infer from Lemma \ref{lem directfin} that, for every $\l\in\L'\cup\L''$, there is a unit $b_{\l}\in R_{\l}$ such that $a_{\l}-b_{\l}\in S_{\l}$. Now the element $a'\in Q$ defined by
\[
a'_{\l} = \left\{
  \begin{array}{ll}
    a_{\l}, & \hbox{if $\l\in\L\setminus(\L'\cup\L'')$;} \\
    b_{\l}, & \hbox{if $\l\in\L'\cup\L''$}
  \end{array}
\right.
\]
is a unit and $a-a'\in S$. Necessarily $a'\in R$ and $a'$ has its inverse in $R$ by the regularity of $R$.
\end{proof}

\begin{pro}\label{pro eR oplus R iso R}
Let $e$ be an idempotent of a ring $R$. Then $eR\oplus R \is R$, as right $R$-modules, if and only if there are $a,b\in R$ such that $ab = 1$ and $(1-ba)R \is eR$.
\end{pro}
\begin{proof}
If $eR\oplus R \is R$, then there are $a,b\in R$, $c\in eR$ and $d\in Re$ such that
\begin{equation}\label{equ dc+ba=1}
    \left(%
\begin{array}{cc}
  d & b \\
\end{array}%
\right)\left(%
\begin{array}{c}
  c \\
  a \\
\end{array}%
\right) = dc+ba=1_R
\end{equation}
and
\begin{equation}\label{equ cd+ba=1bis}
    \left(%
\begin{array}{c}
  c \\
  a \\
\end{array}%
\right)\left(%
\begin{array}{cc}
  d & b \\
\end{array}%
\right) =
\left(%
\begin{array}{cc}
  cd & cb \\
  ad & ab \\
\end{array}%
\right) =
\left(%
\begin{array}{cc}
  e & 0 \\
  0 & 1 \\
\end{array}%
\right) = 1_{eR\oplus R},
\end{equation}
therefore $ab = 1$, $cd = e$ and $cb = 0 = ad$. There is a well defined $R$-linear map $\map{f}{dcR=(1-ba)R}{eR}$ such that $f(dcr) = ecr$ for all $r\in R$; indeed, if $dcr = 0$, then it follows from \eqref{equ dc+ba=1} that $r = bar$, therefore $ecr = ecbar = 0$. If $ecr = 0$, then $dcr = decr = 0$ and therefore $f$ is injective. On the other hand we have
\[
er = cdr = ecdr = f(dcdr)
\]
for all $r\in R$, proving that $f$ is surjective and hence an isomorphism.

Conversely, suppose that $a,b\in R$ exist such that $ab = 1$, consider the idempotent $e' = 1-ba$ and assume that $e'R \is eR$. Then it is immediate to check that each of the matrices
$\displaystyle\left(
                \begin{array}{cc}
                  e' & b \\
                \end{array}
              \right)$ and
$\displaystyle\left(
                \begin{array}{c}
                  e' \\
                  a \\
                \end{array}
              \right)$
is the inverse of the other. As a result we have the isomorphisms
\[
eR\oplus R \is e'R\oplus R \is R.
\]
\end{proof}

\begin{lem}\label{lem decompsimple}
Let $R$ be a regular and semiartinian ring and let us consider an idempotent $x\in R$. If $\a+1 = h(x)$, then there are pairwise orthogonal idempotents $x_{1},\ldots,x_{n}\in xRx$ such that $x = x_{1}+\cdots+x_{n}$ and  $x_{i}R/x_{i}L_{\a}$ is simple for every $i = 1,\ldots,n$.
\end{lem}
\begin{proof}
Since $xR/xL_{\a}$ is a semisimple right ideal of $R/L_{\a}$, then there are pairwise orthogonal idempotents $y_{1},\ldots,y_{n}\in R/L_{\a}$ such that $x+L_{\a} = y_{1}+\cdots+y_{n}$ and each $y_{i}(R/L_{\a})$ is simple. It follows from \cite[Proposition 2.18]{Good:3} that there are pairwise orthogonal idempotents $x_{1},\ldots,x_{n}\in xRx$ such that $x = x_{1}+\cdots+x_{n}$ and $y_{i} = x_{i}+L_{\a}$ for every $i$.
\end{proof}

We are now ready to prove the main theorem.

\begin{theo}\label{theo main}
Let $R$ be a regular and semiartinian ring with Loewy length $\xi+1$. Then the following conditions are equivalent:\begin{enumerate}
             \item $R$ is unit-regular.
             \item Every factor ring of $R$ is directly-finite.
             \item $\KO(R)$ is free with a simp-basis.
           \end{enumerate}
\end{theo}
\begin{proof}
(1)$\Rightarrow$(3): see \cite[Theorem 1, (ii)]{BacCiamp:14}.

(3)$\Rightarrow$(2). Assume (3) and let us prove first that, consequently, $R$ is directly finite. Let $[\ZA] = \{[A]\mid A\in\ZA\}$ be a simp-basis for $\KO(R)$, for some $\ZA\sbs\ESFP(R_R)$, let $a,b\in R$ be such that $ab = 1$, set $u = 1-ba = u^{2}$ and assume that $u\ne 0$. Clearly $a,b\in R\setminus L_{\xi}$ and, since $R/L_{\xi}$ is semisimple and hence directly finite, then $u\in L_{\xi}$. According to Proposition \ref{pro eR oplus R iso R} we have that $[uR] = 0$ in $\KO(R)$. On the other hand, by setting $\a+1 = h(u)$, we see that there are distinct $A_{1},\ldots,A_{n}\in\ZA_{\a}$ and positive integers $k_{1},\ldots,k_{n}$ such that
\[
(uR/uL_{\a}) \is (A_{1}/A_{1}L_{\a})^{k_{1}} \oplus\cdots\oplus (A_{n}/A_{n}L_{\a})^{k_{n}}.
\]
As a result we have that
\[
0 = k_{1}[A_{1}/A_{1}L_{\a}]+\cdots+k_{n}[A_{n}/A_{n}L_{\a}]
\]
in $\KO(R/L_{\a})$ and hence a contradiction, because $[A_{1}/A_{1}L_{\a}],\ldots,[A_{n}/A_{n}L_{\a}]$ are $\BZ$-linearly independent by Proposition \ref{pro simpgenKO}, (2). Thus $ba = 1$, proving that $R$ is directly finite. As a consequence, again by Proposition \ref{pro simpgenKO}, (2) we have that the ring $R/L_{\a}$ is directly finite for every ordinal $\a\le\xi+1$.

Now, given any ideal $J\ne 0$ of $R$, our goal is to show that the group $\KO(R/J)$ has a simp-basis; it will follow that $R/J$ is directly finite by the above. To this purpose, we may assume that the simp-basis $[\ZA]$ for $\KO(R)$ satisfies the following property: if $U\in\Simp_{R}$ and $UJ = U$, then the unique $A\in\ZA$ such that $A/AL_{h(A)-1}\is U$ satisfies the condition $A = AJ$. In fact, we have that $A/AL_{h(A)-1} = (A/AL_{h(A)-1})J$ if and only if $A = AJ+AL_{h(A)-1}$. If it is the case and $A\ne AJ$, then we can choose $x\in AJ$ in such a way that $xR/xL_{h(A)-1} \is U$. Thus $xR\in\ESFP_{R}$, $xR = xJ$ and we replace $A$ with $xR$. After all such replacements we obtain a new subset $\ZB\sbs\ESFP_{R}$ such that $[B]$ is a simp-basis for $\KO(R)$, according to (1) and (3) of Proposition \ref{pro simpgenKO}, which satisfies the required property.

Let $\map{\f}{R}{R/J}$ be the canonical projection and note that the set
\[
[\ZA_{J}] \defug \{[A]\mid A\in\ZA\text{ and }A = AJ\}
\]
cannot be empty, otherwise $J$ would annihilate every $U\in\Simp_{R}$. We claim that
\begin{equation}\label{eq simpbasisfactor}
\langle[\ZA_{J}]\rangle = \Ker(\KO(\f)).
\end{equation}
 By \cite[Proposition 15.15]{Good:3} the first member is contained in the second. By taking Proposition \ref{pro simpgenKO}, (1) into account, in order to prove the opposite inclusion it will be sufficient to show that, given an ordinal $\a$, the inclusion
\[
(**)_{\a}\hspace{27mm}\Ker(\KO(\f))\cap\Ker(\KO(\f_{\a})) \sbs \langle[\ZA_{J}]\rangle \hspace{35,5mm}
\]
holds. Since $(**)_{0}$ is obviously true, let $\a>0$ and assume that $(**)_{\b}$ is true for all $\b<\a$. If $\a$ is a limit ordinal, then $(**)_{\a}$ easily follows from Proposition \ref{pro simpgenKO}, (1. Suppose that $\a = \b+1$ for some $\b$. According again to \cite[Proposition 15.15]{Good:3}, all we have to show is that if $x = x^{2}\in J\cap L_{\a} = JL_{\a}$, then $[xR]\in\langle[\ZA_{J}]\rangle$. Set $X = xR$ and suppose first that $X/XL_{\b}$ is simple. Then there is a unique $A\in\ZA$ such that $X/XL_{\b} \is A/AL_{\b}$ and, since $X/XL_{\b} = (X/XL_{\b})J$, we have that $A = AJ$ by the assumption on $[\ZA]$. It follows from \cite[Proposition 2.19]{Good:3} that there are decompositions $X = X'\oplus X''$ and $A = A'\oplus A''$ such that $X'\is A'$, $X'' = X''L_{\b}$ and $A'' = A''L_{\b}$. Consequently $A' = A'J$ and $X/XL_{\b} \is X'/X'L_{\b} \is A'/A'L_{\b}$, therefore
\[
\KO(\f_{\b})([X] - [A']) = [X/XL_{\b}] - [A'/A'L_{\b}] = 0
\]
and so $[X] - [A']\in \Ker(\KO(\f))\cap\Ker(\KO(\f_{\b}))$. We infer from the inductive hypothesis that $[X] - [A']\in \langle[\ZA_{J}]\rangle$. Inasmuch as  $A = AJ$, then $A'' = A''J$ as well and hence $[A'']\in\Ker(\KO(\f))\cap\Ker(\KO(\f_{\b}))$,
so that $[A'']\in\langle[\ZA_{J}]\rangle$ again by the inductive hypothesis. It follows that $[A'] = [A]-[A'']\in\langle[\ZA_{J}]\rangle$ and we conclude that $[X]\in\langle[\ZA_{J}]\rangle$. Next, assume that $X/XL_{\b}$ is not simple. Then it follows from Lemma \ref{lem decompsimple} that there is a decomposition
\[
X = X_{1}\oplus\cdots\oplus X_{n},
\]
where each $X_{i}/X_{i}L_{\b}$ is simple. As a result
\[
[X] = [X_{1}]+\cdots+[X_{m}]\in\langle[\ZA_{J}]\rangle,
\]
as wanted. Since our claim \eqref{eq simpbasisfactor} is now proved, we may state that the set
\[
\{[A/AJ]\mid A\in\ZA\text{ and }A\ne AJ\}
\]
is a simp-basis for $\KO(R/J)\is\KO(R)/\Ker(\KO(\f))$.

(2)$\Rightarrow$(1). Assume (2) and, as a first step, let us prove that every unit of $R/L_{\a}$ lifts to a unit of $R$, for every ordinal $\a\le\xi$. This is obvious if $\a = 0$, thus, suppose that $0<\a$, assume that every unit of $R/L_{\b}$ lifts to a unit of $R$ whenever $\b<\a$ and let $x+L_{\a}$ be a unit of $R/L_{\a}$. If $\a = \b+1$ for some $\b$, then $(x+L_{\b})+L_{\a}/L_{\b}$ is a unit of $(R/L_{\b})/(L_{\a}/L_{\b})$. By the assumption (2) every factor ring of $R/L_{\b}$ is directly finite, thus it follows from Lemma \ref{lem directfin1} that there exists a unit $a+L_{\b}$ of $R/L_{\b}$ such that $(x+L_{\b})+L_{\a}/L_{\b} = (a+L_{\b})+L_{\a}/L_{\b}$. By the inductive hypothesis there is a unit $b$ of $R$ such that $a+L_{\b} = b+L_{\b}$; consequently $(x+L_{\b})+L_{\a}/L_{\b} = (b+L_{\b})+L_{\a}/L_{\b}$ and therefore $x+L_{\a} = b+L_{\a}$, as wanted. Next, suppose that $\a$ is a limit ordinal and let $y\in R$ be such that both $xy-1$ and $yx-1$ belong to $L_{\a}$. If $\b = \max(h(xy-1),h(yx-1))$, then $\b<\a$ and $x+L_{\b}$ is a unit of $R/L_{\b}$. Again from the inductive hypothesis we have that $x+L_{\b} = b+L_{\b}$ for some unit $b$ of $R$; thus $x+L_{\a} = b+L_{\a}$ as well and the induction is complete.

We now proceed to prove (1) by induction on the ordinal $\xi$. If $\xi = 0$, then $R$ is semisimple and so it is unit-regular. Given an ordinal $\xi>0$, suppose that the implication (2)$\Rightarrow$(1) holds whenever $R$ has Loewy length less than $\xi+1$ and let $R$ be a regular and semiartinian ring with Loewy length $\xi+1$ and having all factor rings directly-finite. As we have shown above, every unit of $R/L_{\xi}$ lifts to a unit of $R$; since $R/L_{\xi}$ is semisimple and hence unit-regular, according to Vasershtein criterion (see \cite[Proposition 4.12]{Vasershtein:1}, or \cite[Lemma 3.5]{Bac:12} for a ready-to-use version), in order to prove that $R$ is unit-regular it is sufficient to show that the ring $eRe$ is unit-regular for every idempotent $e\in L_{\xi}$. First recall that if $R$ is any regular and directly finite ring, then so is $eRe$ for every idempotent $e\in R$. Next, given an idempotent $e\in L_{\xi}$, it follows from \cite[Proposition 1.5]{Bac:15} that $eRe$ is semiartinian with Loewy length at most $\xi$. If $J$ is any ideal of $eRe$ and $I = RJR$, then $J = eIe$ and there is a ring isomorphism $eRe/J\is (e+I)(R/I)(e+I)$. Consequently $eRe/J$ is directly finite, because so is $R/I$ by the assumption. By the inductive hypothesis $eRe$ is unit-regular and the proof is complete.
\end{proof}

\section{The natural partial order of $\Simp_{R}$ versus the natural preorder of $\KO(R)$.}

The Grothendieck group $\KO(R)$ has a canonical structure of a pre-ordered abelian group, where the positive cone is given by the submonoid generated by the set $\langle [P]\mid \text{$P\in \FP_{R}$}\rangle$. Inasmuch as $R$ is a regular and semiartinian ring, then the assignment $U\mapsto r_{R}(U)$ defines a bijection from $\Simp_{R}$ to the set $\Prim_{R}$ of right primitive ideals of $R$, so the ordering by inclusion of $\Prim_{R}$ induces what we called in \cite{Bac:15} the \emph{natural partial order} of $\Simp_{R}$; precisely, if $U,V\in\Simp_{R}$, we set $V\preccurlyeq U$ precisely when $r_{R}(V)\sbs r_{R}(U)$. It appears quite natural to ask about any relationship between the above natural pre-order of $\KO(R)$ and the natural partial order of $\Simp_{R}$; in this section we will give some partial answers to that question. In view of our goals, it will be convenient to have at disposal a couple of new concepts.

Let us consider an element $A\in\ESFP_{R}$, set $\a+1 = h(A)$ and let us consider the simple module $U = A/AL_{\a}$\,. Given any $V\in\Simp_{R}$ with $h(V) = \b+1<\a+1$, it follows from \cite[Theorem 2.1]{Bac:16} \footnote{We take here the opportunity to remark that a mistake has been left to stand in the final and published version of the quoted paper. Precisely, in the quoted theorem and its proof ten instances appear of a factor module of the form $(***+L_{\b+1})/L_{\b}$, which should be read as $(***+L_{\b})/L_{\b}$ instead, as one can easily infer from the context.} that $V\prec U$ if and only if $A/AL_{\b}$ contains an infinite direct sum of copies of $V$. Let us say that $A$ is \emph{order-selective} if, for every $V\in\Simp_{R}$, the property $V\lesssim A/AL_{h(V)-1}$ is equivalent to $V\preccurlyeq U$. Note that $A$ is order-selective if and only if, given $C\in\ESFP_{R}$, the property $C\lesssim A$ implies that $C/CL_{h(C)-1}\preccurlyeq U$.

\begin{lem}\label{lem extendedtrace}
Let $R$ be a regular and semiartinian ring, let $A\in\ESFP_{R}$, $B\in\FP_{R}$ be such that $h(B)<h(A)$, assume that $A$ is order-selective and the following condition holds:
\begin{enumerate}
  \item If $S\in\Simp_{R}$, the condition $S\lesssim B/BL_{h(S)-1}$ implies $S\prec A/AL_{h(A)-1}$.
\end{enumerate}
Then $B\lesssim A$.
\end{lem}
\begin{proof}
Since the thesis is obvious if $B = 0$, we may assume that $B\ne 0$. We have that $B/BL_{h(B)-1} \is S_{1}^{k_{1}}\oplus\cdots\oplus S_{n}^{k_{n}}$ for some $S_{1},\ldots, S_{n}\in\Simp_{R}$, with $h(S_{1}) = \cdots = h(S_{n}) = h(B)$ and positive integers $k_{1},\ldots,k_{n}$. Then $S_{i}\prec A/AL_{h(A)-1}$ for all $i = 1,\ldots,n$ by the assumption (1) and therefore $B/BL_{h(B)-1}\lesssim A/AL_{h(B)-1}$ by Theorem 2.1, (2) of \cite{Bac:16}. As a result, the thesis is true in case $h(B) = 1$. Set $h(A) = \a+1$ and, given an ordinal $\b$ such that $0<\b<\a$, assume that the thesis is true whenever $h(B)<\b+1$ and suppose that $h(B)=\b+1$. As we noted above, we have that $B/BL_{\b}\lesssim A/AL_{\b}$, thus there are two decompositions
\[
A = A'\oplus A'', \qquad B = B'\oplus B''
\]
such that $B'\is A'$ and $B'' = B''L_{\b}$. Note that $A''\in\ESFP_{R}$, because $h(A'') = \a+1$ and $A''/A''L_{\a}\is A/AL_{\a}$. Moreover $A''$ is order-selective. Indeed, suppose that $C\lesssim A''$ for some $C\in\ESFP_{R}$. Then $C\lesssim A$ and therefore $C/CL_{h(C)-1}\preccurlyeq A/AL_{\a} \is (A'/A'L_{\a})\oplus(A''/A''L_{\a}) = A''/A''L_{h(A'')-1}$. In addition we have that $h(B'')<\b+1<\a+1= h(A'')$ and the property (1) holds relative to $B''$ and $A''$. Indeed, suppose that $S\lesssim B''/B''L_{h(S)-1}$ for some $S\in\Simp_{R}$. Then $S\lesssim B/BL_{h(S)-1}$ and so $S\prec A/AL_{\a}\is A''/A''L_{h(A'')-1}$. The inductive assumption implies now that $B''\lesssim A''$ and we conclude that $B\lesssim A$.
\end{proof}

The following corollaries will play a crucial role when we will describe in an explicit way the positive cone $\KO(R)^{+}$ of $\KO(R)$ in the next section.

\begin{cor}\label{cor orderselect1}
Let $R$ be a semiartinian and regular ring, let $A,B\in\ESFP_{R}$ be order-selective and suppose that $h(B)<h(A)$. Then $B\lesssim A$ if and only if $B/BL_{h(B)-1}\preccurlyeq A/AL_{h(A)-1}$.
\end{cor}
\begin{proof}
We already observed that the ``only if'' part holds as a consequence of the order-selectivity of $A$. Conversely, assume that $B/BL_{h(B)-1}\preccurlyeq A/AL_{h(A)-1}$. Given $S\in\Simp_{R}$, if $S\lesssim B/BL_{h(S)-1}$, then $S\preccurlyeq B/BL_{h(B)-1}$ by the order-selectivity of $B$, consequently $S\preccurlyeq A/AL_{h(A)-1}$ by the transitivity of $\preccurlyeq$. Thus $B\lesssim A$ by Lemma \ref{lem extendedtrace}.
\end{proof}

We say that a complete and irredundant subset $\ZA\sbs \ESFP_{R}$ is \emph{order-representa\-tive} if all of its elements are order-selective; if it is the case, then Corollary \ref{cor orderselect1} tells us that the following property holds:
\begin{equation}\label{eq order-selective}
    \text{given $A,B\in\ZA$, then $A\lesssim B$ if and only if $A/AL_{h(A)-1}\preccurlyeq B/BL_{h(B)-1}$}.
\end{equation}
If $\ZA$ is order-representative then, with respect to the algebraic preorder, the subset $\{\overline{A}\mid A\in\ZA\}$ of $\CV(R)$ is order isomorphic to $\Simp_{R}$ and hence it is partially ordered.

\begin{cor}\label{cor orderselect2}
Let $R$ be a semiartinian and regular ring and let $\ZA$ be an order-representative subset of $\ESFP_{R}$. If $A, B_{1},\ldots,B_{r}\in\ZA$ are such that $B_{i}\lesssim A$ and $h(B_{i})<h(A)$ for all $i\in\{1,\ldots,r\}$, then
\begin{equation}\label{eq orderselect3}
B_{1}^{n_{1}}\oplus\cdots\oplus B_{r}^{n_{r}}\lesssim A
\end{equation}
for every choice of positive integers $n_{1},\ldots,n_{r}$.
\end{cor}
\begin{proof}
Let us consider first the case $r = 1$. Given $S\in\Simp_{R}$, suppose that
\[
S\lesssim B_{1}^{n_{1}}/B_{1}^{n_{1}}L_{h(S)-1} \is (B_{1}/B_{1}L_{h(S)-1})^{n_{1}}.
\]
Then $S\lesssim B_{1}/B_{1}L_{h(S)-1}$ and therefore $S\lesssim A/AL_{h(S)-1}$. Inasmuch as $A$ is order selective, it follows that $S\preccurlyeq A/AL_{h(S)-1}$ and, consequently, $B_{1}^{n_{1}}\lesssim A$ by Lemma \ref{lem extendedtrace}.

Next, let $r>1$ and assume, inductively, that
$B_{1}^{n_{1}}\oplus\cdots\oplus B_{r-1}^{n_{r-1}}\lesssim A$.
Then
\[
A\is A'\oplus B_{1}^{n_{1}}\oplus\cdots\oplus B_{r-1}^{n_{r-1}}
\]
for a suitable $A'\in\FP_{R}$. We note that $h(A') = h(A)$ and $A'/A'L_{h(A')-1}\is A/AL_{h(A)-1}$; moreover $A'$ is order-selective. Since $B_{r}/B_{r}L_{h(B_{r})-1}\preccurlyeq A/AL_{h(A)-1}\is A'/A'L_{h(A')-1}$, we infer from Corollary \ref{cor orderselect1} that $B_{r}\lesssim A'$ and the first part of the present proof tells us that $B_{r}^{n_{r}}\lesssim A'$. We conclude that \eqref{eq orderselect3} holds.
\end{proof}

Next, let $M$ be a right module over a ring $R$ (here we make no assumption on $R$, apart multiplicative identity) and let us consider the set of ideals
\[
\CH_{M} \defug \{J\in\BL_{2}(R)\mid MJ = M\}.
\]
Obviously $\CH_{M}$ is not empty, as $R\in\CH_{M}$. If $\CH_{M}$ has a minimal element $H$, then $H$ is the minimum, because if $J\in\CH_{M}$, then $M(JH) = M$ and therefore $JH = H$, from which $H\sbs J$. The same argument shows also that $H = H^{2}$. If it exists, we shall denote with $H_{M}$ the minimum of $\CH_{M}$. As it is standard, we denote by $\Tr_{R}(M)$ the trace ideal of $M$ in $R$; it is immediate to check that $\Tr_{R}(M)\sbs H_{M}$ and the equality holds when $M$ is projective. We call $H_{M}$ the \emph{extended trace} of $M$. In our present context of regular and semiartinian rings we are interested in the extended traces of simple modules.

\begin{pro}\label{pro extendedtrace}
Let $R$ be a semiartinian and regular ring, let $U,V$ be simple right $R$-modules and set $\a+1 = h(U)$, $\b+1 = h(V)$. Then the following properties hold:
\begin{enumerate}
\item If it exists, then $H_{U}$ is the smallest ideal of $R$ such that
\begin{equation}\label{eq extendedtrace}
\Tr_{R/L_{\a}}(U) = (H_{U}+L_{\a})/L_{\a}.
\end{equation}
In addition there exists an idempotent $x\in H_{U}$ such that $H_{U} = RxR$ and $xR/xL_{\a}\is U$.
\item If $H_{U}$ exists and $U\preccurlyeq V$, then $H_{U}$ is contained in every ideal $K$ of $R$ such that $VK = V$.
\item If $H_{V}$ exists, then
\[
U\not\preccurlyeq V\quad\text{ if and only if }\quad UH_{V} = 0.
\]
Consequently, if $H_{U}$ exists too, then
\[
U\preccurlyeq V\quad\text{ if and only if }\quad H_{U}\sbs H_{V}.
\]
\end{enumerate}
\end{pro}
\begin{proof}
(1) First recall that, over a semiprime ring, the trace of a simple and projective right module is a minimal ideal. In the present setting, if $U\in\Simp_{R}$ and $\a+1 = h(U)$, then $U$ is $R/L_{\a}$-projective. Let $J$ be the unique ideal of $R$ such that $\Tr_{R/L_{\a}}(U) = J/L_{\a}$. We have that $H_{U}\nsbs L_{\a}$, otherwise it would follow that $UH_{U} = 0$; thus, since $J/L_{\a}$ is a minimal ideal of $R/L_{\a}$, then $J/L_{\a} = (J/L_{\a})((H_{U}+L_{\a})/L_{\a})$ and it follows that $J/L_{\a}\sbs (H_{U}+L_{\a})/L_{\a}$. On the other hand we have that $U = UJ$, so that $H_{U}\sbs J$ and \eqref{eq extendedtrace} is proved. If $K$ is any ideal of $R$ such that
\[
\Tr_{R/L_{\a}}(U) = (K+L_{\a})/L_{\a}.
\]
then $U = UK$ and therefore $H_{U}\sbs K$. Now there is an idempotent $x\in H_{U}$ such that $xR/xL_{\a}\is U$ and $Ux \ne 0$. As a consequence $U(RxR) = U$ and so $H_{U} = RxR$.

(2) By the same argument of the proof of (1), we have that $(K+L_{\b})/L_{\b}$ contains $\Tr_{R/L_{\b}}(V)$ and, consequently, there exists $y\in K$ such that $(yR+L_{\b})/L_{\b} \is V$. By \cite[Theorem 2.1]{Bac:16} there is some $x\in R$ such that $(xR+L_{\a})/L_{\a} \is U$ and $RxR\sbs RyR$. Since $U(RxR) = U$, it follows from (1) that $H_{U}\sbs RxR\sbs RyR\sbs K$.

(3) If $U\not\preccurlyeq V$, then $V\cdot r_{R}(U) = V$ and hence $H_{V}\sbs r_{R}(U)$. Conversely, the latter condition implies that $V\cdot r_{R}(U) \ne 0$ and so $U\not\preccurlyeq V$.
\end{proof}

\begin{cor}\label{cor extende3dtrace}
Let $R$ be a semiartinian and regular ring. Given an ordinal $\a$, if $H_{U}$ exists for all $U\in\Simp_{R}$ such that $h(U)\le\a$, then
\begin{equation}\label{eq extendedtrace1}
L_{\a} = \sum\{H_{U}\mid U\in\Simp_{R}\text{ and }h(U)\le\a\}.
\end{equation}
\end{cor}
\begin{proof}
In view of Proposition \ref{pro extendedtrace}, (1), the equality \eqref{eq extendedtrace1} is clear if $\a = 1$. Let $\a>1$ and assume that
\begin{equation}\label{eq extendedtrace2}
L_{\b} = \sum\{H_{U}\mid U\in\Simp_{R}\text{ and }h(U)\le\b\}
\end{equation}
for every $\b<\a$. Then \eqref{eq extendedtrace1} immediately follows in case $\a$ is a limit ordinal. Suppose that $\a = \b+1$ for some $\b$. Then we have:
\begin{align*}
L_{\a}/L_{\b} &= \Soc(R/L_{\b}) = \sum\left.\left\{\Tr_{R/L_{\b}}(U)\right|U\in\Simp_{R}\text{ and }h(U)=\a\right\} \\
&= \sum\left.\left\{(H_{U}+L_{\b})/L_{\b}\right|U\in\Simp_{R}\text{ and }h(U)=\a\right\}\,\,\,\,\,\text{ by \eqref{eq extendedtrace}} \\
&= \left.\left(\sum\{H_{U}\mid U\in\Simp_{R}\text{ and }h(U)=\a\}+L_{\b}\right)\right/L_{\b}.
\end{align*}
Then \eqref{eq extendedtrace1} follows.
\end{proof}

If $R$ is a semiartinian and regular ring, then there is a deep link between the lattice of ideals of $R$ and the order structure of $\Simp_{R}$; in fact, if $H$ is an ideal of $R$, then $\Simp_{R/H}$ is an upper subset of
$\Simp_{R}$, so that we may consider the decreasing map
\[
\lmap{\F}{\BL_2(R)}{\Uparrow\!\!\Simp_{R}}
\]
defined by $\F(H) = \Simp_{R/H}$. This map is injective and has as a left inverse the map
\[
\lmap{\P}{\Uparrow\!\!\Simp_{R}}{\BL_2(R)}
\]
defined by $\P(\ZS) = \bigcap\{r_R(U)\mid U\in\ZS\}$. Indeed, it is clear that $\F(H)\sps\F(K)$ whenever $H\sbs K$. Inasmuch as $R$ is
regular, then every ideal of $R$ is the intersection of all
primitive ideals containing it. Thus, given $H\in\BL_2(R)$, we have
\begin{gather*}
\P(\F(H)) = \P\left(\Simp_{R/H}\right) = \bigcap\left\{r_R(U)\mid U\in\Simp_{R/H}\right\} \\
= \bigcap\left\{r_R(U)\mid \text{$U\in\Simp_{R}$ and $UH = 0$}\right\} = H.
\end{gather*}
In general the map $\F$ need not be surjective. In \cite[Definition 2.5]{Bac:16} we defined $R$ to be \emph{very well behaved} in case the map $\F$ is an anti-isomorphism. As already observed in \cite{Bac:16}, if $\Simp_{R}$ has no infinite antichains, then $R$ is very well behaved; moreover, if $R$ is very well behaved, then so is every factor ring of $R$ and $\Simp_{R}$ has a finite cofinal subset (see \cite[Proposition 2.7]{Bac:16}).  With the next result we establish a strict relationship between the property of being very well behaved, the existence of the extended trace for every member of $\Simp_{R}$ and the existence of a simp-generating set $[\ZA]$ for $\KO(R)$, where $\ZA$ is an order representative subset of $\ESFP_{R}$.

We shall need the following lemma.

\begin{theo}\label{theo extendedtrace}
Let $R$ be a regular and semiartinian ring with Loewy length $\xi+1$. Then the following conditions are equivalent:\begin{enumerate}
             \item $R$ is very well behaved.
             \item Every simple right $R$-module has the extended trace.
             \item $\ESFP_{R}$ has an order-representative subset.
           \end{enumerate}
\end{theo}
\begin{proof}
(1)$\Leftrightarrow$(2). Let us consider the commutative diagram of maps
\begin{equation*}
\begin{split}
\begin{picture}(120,70)(0,0)
\putv{16}{37}{4}{1}{78}
\putv{94}{52}{-4}{-1}{78}
\putv{16}{28}{4}{-1}{78}
\putv{94}{4}{-4}{1}{78}
\putbb{0}{25}{$\BL_{2}(R)$}
\putbb{108}{25}{$\scriptstyle{(-)'}$}
\putbb{132}{25}{$\scriptstyle{(-)'}$}
\putbb{120}{2}{$\Uparrow\!\Simp_{R}$}
\putbb{145}{2}{,}
\putbb{120}{50}{$\Downarrow\!\Simp_{R}$}
\putv{118}{47}{0}{-1}{32}
\putv{122}{15}{0}{1}{32}
\putbb{54}{48}{$\scriptstyle\G$}
\putbb{54}{35}{$\scriptstyle\D$}
\putbb{54}{7}{$\scriptstyle\P$}
\putbb{54}{20}{$\scriptstyle\F$}
\end{picture}
\end{split}
\end{equation*}
where
\begin{align*}
\G(H) &= \{U\in\Simp_{R}\mid UH = U\}, \\
\P(\ZS) &= r_{R}(\ZS), \\
\ZS' &= \Simp_{R}\setminus\ZS, \\
\D(\ZS) &= \P(\ZS')
\end{align*}
for all $H\in\BL_{2}(R)$ and $\ZS\sbs\Simp_{R}$. Since every ideal of $R$ is the intersection of primitive ideals, it is true that
\[
\D\G = 1_{\BL_{2}(R)} = \P\F
\]
and $R$ is very well behaved if and only if $\G$ is bijective, that is $\G\D = 1_{\Downarrow\!\Simp_{R}}$. If it is the case, for every $U\in\Simp_{R}$ let us consider the ideal
\[
H \defug \D(\{\preccurlyeq U\}) = \Psi(\{\preccurlyeq U\}') = \bigcap\{r_{R}(V)\mid V\not\preccurlyeq U\}.
\]
Then
\[
U\in \{\preccurlyeq U\} = \G(\D(\{\preccurlyeq U\})) = \G(H) = \{V\in\Simp_{R}\mid VH = V\},
\]
therefore $UH = U$. Assume that $UK = U$ for some $K\in\BL_{2}(R)$ and set $\ZS = \Simp_{R/K}$. Since $\ZS'$ is a lower subset of $\Simp_{R}$ and $U\in\ZS'$, then $\{\preccurlyeq U\}\sbs\ZS'$. As a result we get
\[
H = \D(\{\preccurlyeq U\}) \sbs \D(\ZS') = \P(\ZS) = K.
\]
This shows that $H$ is the extended trace of $U$. Conversely, assume that every member of $\Simp_{R}$ has the extended trace and, given any lower subset $\ZS$ of $\Simp_{R}$, let us consider the ideal
\[
K = \sum\{H_{U}\mid U\in\ZS\}.
\]
Then $UK = U$ for all $U\in\ZS$ and, if $V\in\Simp_{R}$ is such that $VK = V$, then there is some $U\in\ZS$ such that $VH_{U} = V$. This implies that $H_{V}\sbs H_{U}$, therefore $V\preccurlyeq U$ by Proposition \ref{pro extendedtrace} and hence $V\in\ZS$. We conclude that $\ZS = \G(K)$, so that $\G$ is bijective.

(2)$\Rightarrow$(3). Assume that every $U\in\Simp_{R}$ has its extended trace $H_{U}$ and, accor\-ding to Proposition \ref{pro extendedtrace}, for each $U\in\Simp_{R}$ let us choose an idempotent $x_{U}\in H_{U}$ such that $x_{U}R/x_{U}L_{h(U)-1} \is U$ and $H_{U} = Rx_{U}R$. Then $\ZA = \{x_{U}R\mid U\in\Simp_{R}\}$ is a complete and irredundant subset of $\ESFP_{R}$; we claim that it is also order representative. Given $U, V\in\Simp_{R}$, set $\a+1 = h(U)$, $\b+1 = h(V)$, assume that $\a<\b$ and suppose that $U\not\prec V$. Then $UH_{V} = 0$ by Proposition \ref{pro extendedtrace}. If $U\lesssim x_{V}R/x_{V}L_{\a}$, since $x_{V}R/x_{V}L_{\a} = (x_{V}R/x_{V}L_{\a})H_{V}$ and $R$ is regular, we get that $UH_{V} = U$: a contradiction. Thus $U\not\lesssim x_{V}R/x_{V}L_{\a}$ and therefore $x_{U}R$ is order-selective.

Suppose now that $U\prec V$ and let us prove that, consequently, $x_{U}R\lesssim x_{V}R$. If $S\in\Simp_{R}$ and $S\lesssim x_{U}R/x_{U}L_{h(S)-1}$, then $S = S(Rx_{U}R) = SH_{U}$ and therefore, by using Proposition \ref{pro extendedtrace} we see that $S\prec U\prec V$. Thus we can apply Lemma \ref{lem extendedtrace} to $B = x_{U}R$ and $A = x_{V}R$ in order to obtain that $x_{U}R\lesssim x_{V}R$.

(3)$\Rightarrow$(2). Assume that $\ESFP_{R}$ has an order-representative subset $\ZA$ and let $U\in\Simp_{R}$. Then there is a unique $A\in\ZA$ such that $A/AL_{h(U)-1}\is U$. In turn, since $A$ is isomorphic to a finite direct sum of principal right ideals of $R$, then there is an idempotent $e\in L_{h(U)}$ such that $eR\lesssim A$ and $eR/eL_{h(U)-1}\is U$. Moreover, since $A$ is order-selective, then $eR$ is order-selective as well. We claim that $H = ReR$ is t5he extended trace of $U$. First, it is clear that $U = UH$. Thus, in order to establish our claim, it is sufficient to prove that $H$ is minimal among those ideals $I$ of $R$ such that $U = UI$. Let $K$ be an ideal of $R$ contained in $H$ such that $U = UK$ and note that, consequently, there is an idempotent $f\in K$ such that $fR/fL_{h(U)-1}\is U$. In order to prove that $H = K$ it is sufficient to show that
\begin{equation}\label{eq extendedtrace3}
    H\cap L_{\a} = K\cap L_{\a}
\end{equation}
for every ordinal $\a\le h(U)$ and, by the assumption on $K$, we only have to show that the first member is contained in the second. Since this is obvious if $\a = 0$, let us consider an ordinal $\a>0$ and suppose, inductively, that $H\cap L_{\b} = K\cap L_{\b}$ whenever $\b<\a$, If $\a$ is a limit ordinal, then \eqref{eq extendedtrace3} immediately follows. Assume that $\a = \b+1$ for some $\b$, let $x\in H\cap L_{\a}$ be such that $x\nin L_{\b}$ and suppose first that $xR/xL_{\b}$ is simple. Since $xR\sbs H = ReR$, it follows from \cite[Corollary 2.23]{Good:3} that $xR\lesssim (eR)^{n}$ for some $n\in\BN$. As a result, since $xR/xL_{\b}$ is a simple and projective right $R/L_{\b}$-module, we infer that $xR/xL_{\b}\lesssim eR/eL_{\b}$. Inasmuch as $eR$ is order-selective, then $xR/xL_{\b}\preccurlyeq U\is fR/fL_{h(U)-1}$ and, accordingly, we have that $xR/xL_{\b}\lesssim fR/fL_{\b}$. By \cite[Proposition 2.20]{Good:3} there are decompositions
\[
xR = x'R\oplus x''R, \qquad fR = f'R\oplus f''R,
\]
for some idempotents $x',x''\in xR$ and $f',f''\in fR$, such that $x'R\is f'R$ and $x''R = x''L_{\b}$. Consequently, from the inductive assumption it follows that $x''\in H\cap L_{\b} = K\cap L_{\b}$. Moreover $Rx'R = Rf'R\sbs K$ and we conclude that $x\in K$. If $xR/xL_{\b}$ is \emph{not} simple then, by Lemma \ref{lem decompsimple}, there are pairwise orthogonal idempotents $x_{1},\ldots,x_{n}\in H$ such that $x = x_{1}+\cdots+x_{n}$ and each $x_{i}R/x_{i}L_{\b}$ is simple. The above argument allows us to conclude that $x\in K$ and the proof is complete.
\end{proof}

\section{The natural partial order of $\KO(R)$ when $R$ is a unit-regular, semiartinian and very well behaved ring.}

In this last section we aim to describe in an explicit way the canonical order structure of the group $\KO(R)$ for a semiartinian, unit-regular and very well behaved ring $R$. As we shall see, in this case $\KO(R)$ turns out to be a dimension group whose order structure is completely determined by the partially ordered set $\Prim_{R}$.

Given a partially ordered set $I$, let us denote by $G(I)$ the free abelian group with $I$ as a basis; we represent each element $x\in G(I)$ as a unique linear combination of elements of $I$ with coefficients in $\BZ$ and, if $i\in I$, we denote by $x_{i}$ the coefficient of $i$. As usual, the \emph{support} of an element $x\in G(I)$ is the set $\Supp(x)\defug \{i\in I\mid x_{i}\ne 0\}$. Let us consider the following subset of $G(I)$:
\begin{equation}\label{eq M(I)}
M(I) \defug \{x\in G(I)\mid x_{i}>0 \text{ if $i$ is a maximal element of $\Supp(x)$}\}.
\end{equation}
As we are going to see, $M(I)$ is actually a submonoid of $G(I)$. We recall that an abelian monoid $M$ is said to be \emph{conical} in case the property $x+y = 0$ for two elements $x,y\in M$ implies that $x=y=0$. If $G$ is a pre-ordered abelian group with positive cone $G^{+}$, then $G^{+}$ is conical if and only if the pre-order of $G$ is a partial order.

\begin{pro}\label{pro dimgrp}
For every partially ordered set $I$, the subset $M(I)$ of $G(I)$ defined in \eqref{eq M(I)} is a conical submonoid.
\end{pro}
\begin{proof}
It is clear from the definition \eqref{eq M(I)} that $0\in M(I)$. Let $x,y\in M(I)$ and let $i$ be a maximal element of $\Supp(x+y)$. Then $i\in\Supp(x)$ or $i\in\Supp(y)$. We claim that if $i\in\Supp(x)$ (resp. $i\in\Supp(y)$), then $i$ is maximal in $\Supp(x)$ (resp. $\Supp(y)$). Indeed, suppose that there is some $j\in\Supp(x)$ such that $j>i$. We may assume that $j$ is maximal in $\Supp(x)$. Necessarily $j\nin\Supp(x+y)$, therefore $0 = (x+y)_{j} = x_{j}+y_{j}$. Since $x_{j}>0$, then $y_{j}<0$ and this means that $j\in\Supp(y)$ but $j$ is not maximal in $\Supp(y)$. Let $k$ be a maximal element of $\Supp(y)$ such that $k>j$. Necessarily $k\nin\Supp(x+y)$ and $k\nin\Supp(x)$, therefore $0 = (x+y)_{k} = x_{k}+y_{k} = y_{k}>0$, hence a contradiction. Thus $i$ is maximal in $\Supp(x)$. Now, if $i\in\Supp(x)\cap\Supp(y)$, by the above $i$ is maximal both in $\Supp(x)$ and in $\Supp(y)$, therefore $(x+y)_{i} = x_{i}+y_{i}>0$. If $i\in\Supp(x)$ but $i\nin\Supp(y)$, then $(x+y)_{i} = x_{i}>0$; similarly, if $i\in\Supp(y)$ but $i\nin\Supp(x)$, then $(x+y)_{i} = y_{i}>0$. We conclude that $x+y\in M(I)$ and so $M(I)$ is a submonoid of $G(I)$.

In order to prove that $M(I)$ is conical, let us consider the following statement: given a non-negative integer $n$ and a subset $F$ of $I$ with $|F|=n$, if $x\in G(F)$ is such that both $x$ and $-x$ are in $M(F)$, then $x=0$. Since the statement is obvious if $n=0$, let us consider $n>0$ and let us assume that the statement is true for all subsets of $I$ of cardinality less than $n$. Given a subset $F\sbs I$ with $|F|=n$ and an element $x\in G(F)$ such that both $x$ and $-x$ are in $M(F)$, let us consider a maximal element $m$ of $F$. Then we have that $x_{m}\ge 0$ and $-x_{m} = (-x)_{m}\ge 0$, therefore $x_{m}=0$. Consequently $x\in F\setminus\{m\}$ and hence $x=0$ by the inductive hypothesis.
\end{proof}

Throughout the present section we shall consider the group $G(I)$ partially ordered by choosing the submonoid $M(I)$ as positive cone.

Let $M$ be an abelian monoid and let $x_{1},x_{2},y_{1},y_{2}\in M$ be such that
\begin{equation}\label{eq refinement}
    x_{1}+x_{2} = y_{1}+y_{2}.
\end{equation}
We say that the equality \eqref{eq refinement} has a \emph{(Riesz) refinement} if there exists a matrix $\begin{pmatrix} z_{11} & z_{12} \\ z_{21} & z_{22} \end{pmatrix}$ of elements of $M$ such that $x_{\a} = z_{\a1}+z_{\a2}$ and $y_{\b} = z_{1\b}+z_{2\b}$ for each $\a,\b$; the above matrix is called a
\emph{refinement matrix} for \eqref{eq refinement}; it is standard to use the display
\begin{equation}\label{eq refinementt}
\begin{matrix}
 & \begin{matrix}
                      y_{1} & y_{2}
                    \end{matrix}
   \\
  \begin{matrix}
    x_{1} \\
    x_{2}
  \end{matrix}
 \!\!\!\!\!& \begin{pmatrix}
         z_{11} & z_{12} \\
         z_{21} & z_{22} \\
       \end{pmatrix}

\end{matrix}
\end{equation}
in order to express the stated property, by saying that \eqref{eq refinementt} is a refinement for \eqref{eq refinement}. The monoid $M$ is called a \emph{refinement monoid} if every equality of the form \eqref{eq refinement} between elements of $M$ admits a refinement. It is easy to check that $\BN$ is a refinement monoid. If $x_{1},x_{2},y_{1},y_{2}\in\BZ$ satisfy \eqref{eq refinement}, then it is not difficult to see that there are two refinements as in \eqref{eq refinementt}, where in the first one $z_{11}\ge 0$ and $z_{22}\ge 0$, while $z_{12}\ge 0$ and $z_{21}\ge 0$ in the second. It is well known that if $R$ is a regular ring or, more generally, is an exchange ring, then $\CV(R)$ a refinement monoid. Let $G$ be a partially ordered abelian group, with positive cone $G^{+}$. Then $G^{+}$ is a refinement monoid if and only if $G$ satisfies the \emph{Riesz interpolation property}, namely: given $x_{1},x_{2},y_{1},y_{2}\in G$ such that $x_{\a}\le y_{\b}$ for all $\a,\b$, then there is some $z\in G$ such that $x_{\a}\le z\le y_{\b}$ for all $\a,\b$ (see \cite[Proposition 2.1]{Good:10}). If it is the case, then $G$ is called an \emph{interpolation group}.
The group $G$ is said to be \emph{directed} if it is upward directed or, equivalently, downward directed. The group $G$ is said to be \emph{unperforated} if, given $x\in G$ and a positive integer $n$, the condition $nx\in G^{+}$ implies that $x\in G^{+}$.

Finally, $G$ is called a \emph{dimension group} if it is a directed, unperforated interpolation group (see \cite[Ch. 3]{Good:10}).

\begin{pro}\label{pro unperforated}
If $I$ is a partially ordered set, then the partially ordered abelian group $G(I)$ is a dimension group. Moreover $G(I)$ has an order-unit if and only if $I$ has a finite cofinal subset $F$, in which the case the sum of all elements in $F$ is an order-unit.
\end{pro}
\begin{proof}
Since $G(I)$ is partially ordered and, as a group, is generated by $M(I)$ (because $I\sbs M(I)$), then $G(I)$ is directed by \cite[Proposition 1.3]{Good:10}.

Let $x\in G(I)$, let $n$ be a positive integer and suppose that $nx\in M(I)$. Given a maximal element $i$ of $\Supp(ny) = \Supp(y)$, we have that $nx_{i} = (nx)_{i}>0$, therefore $x_{i}>0$. Thus $G(I)$ is unperforated.

In order to prove that $G(I)$ is an interpolation group we will show that $M(I)$ is a refinement monoid. To this purpose it is clearly sufficient to assume that $I$ is finite, so we will proceed by induction on the cardinality of $I$. If $I$ has just one element, then $M(I)$ is order isomorphic to $\BN$ (with the natural ordering) and so is a refinement monoid. Given a positive integer $n$, assume that $M(I)$ is a refinement monoid whenever $|I| = n$, let $I$ be a partially ordered set with $|I| = n+1$ and let $x_{1}, x_{2},y_{1}, y_{2}\in M(I)$ be such that
\begin{equation}\label{eq dimgrp4}
x_{1}+x_{2}=y_{1}+y_{2}.
\end{equation}
Note that the above equality is equivalent to the condition
\[
(x_{1})_{i}+(x_{2})_{i}=(y_{1})_{i}+(y_{2})_{i}, \quad\text{for all $i\in I$}.
\]
Of course, we may assume that $\Supp(x_{1})\cup\Supp(x_{2})\cup\Supp(y_{1})\cup\Supp(y_{2}) = I$. If $I$ is an antichain, that is every element is minimal and maximal, then $M(I)$ is order isomorphic to the product of $n+1$ copies of $\BN$ with the product ordering, which is a refinement monoid. Otherwise, let us choose a minimal element $m$ of $I$ which is not maximal and set $J\defug I\setminus \{m\}$; moreover, for every $x\in G(I)$, set $x'\defug x-x_{m}m$ and note that $x'\in M(J)$. We infer from \eqref{eq dimgrp4} that
\[
x'_{1}+x'_{2}=y'_{1}+y'_{2}.
\]
By the inductive hypothesis the above equality has a refinement
\[
\begin{matrix}
 & \begin{matrix}
                      y'_{1} & y'_{2}
                    \end{matrix}
   \\
  \begin{matrix}
    x'_{1} \\
    x'_{2}
  \end{matrix}
 \!\!\!\!\!& \begin{pmatrix}
         u_{11} & u_{12} \\
         u_{21} & u_{22} \\
       \end{pmatrix},
\end{matrix}
\]
where $u_{11},u_{12},u_{21},u_{22}\in M(J)$. As far as the equality
\[
(x_{1})_{m}+(x_{2})_{m}=(y_{1})_{m}+(y_{2})_{m}
\]
is concerned, it admits always some refinement
\begin{equation}\label{eq refin1}
\begin{matrix}
 & \,\,\,(y_{1})_{m} & (y_{2})_{m}\\
\,\,\,\begin{matrix}
(x_{1})_{m}\\
    (x_{2})_{m}
     \end{matrix} &
     \!\!\!\!\!\left(\,\,\,\,\begin{matrix}a_{11}\\
 a_{21}
  \end{matrix}\right. &
    \left. \,\,\,\,\,\,\,\begin{matrix}a_{12} \\
     a_{22}
     \end{matrix}\,\,\,\right)
 \end{matrix}
    \end{equation}
in $\BZ$. Thus, if we define the element $z_{\a\b}\in G(I)$, for every $\a,\b\in\{1,2\}$, by setting
\[
(z_{\a\b})_{i} = \begin{cases}(u_{\a\b})_{i}, \text{ if $i\in J$} \\
                               a_{\a\b}, \text{ if $i=m$},
                               \end{cases}
\]
then
\[
\begin{matrix}
 & \begin{matrix}
                      y_{1} & y_{2}
                    \end{matrix}
   \\
  \begin{matrix}
    x_{1} \\
    x_{2}
  \end{matrix}
 \!\!\!\!\!& \begin{pmatrix}
         z_{11} & z_{12} \\
         z_{21} & z_{22} \\
       \end{pmatrix}
\end{matrix}
\]
is a refinement in $G(I)$ for the equality \eqref{eq dimgrp4}; however, we are looking for a refinement in $M(I)$. To this purpose, for every ordered pair $(\a,\b)$ let us denote by $P_{\a\b}$ the property
\[
\text{``\,\,if $i\in I$ and $i>m$, then $(u_{\a\b})_{i} = 0$\,\,''}
\]
and observe that if $P_{\a\b}$ is false for some $\a,\b$ and \eqref{eq refin1} is \emph{any} refinement in $\BZ$, then $z_{\a\b}\in M(I)$. Indeed, there is some $i\in I$ such that $i>m$ and $(u_{\a\b})_{i} >0$; consequently, if $j$ is a maximal element of $\Supp(z_{\a\b})$, then $j\ne m$ and therefore $j\in J$. By the inductive assumption $(z_{\a\b})_{j} = (u_{\a\b})_{j}>0$. As a first consequence, if $P_{\a\b}$ is false for \emph{all} $\a,\b$ then \eqref{eq refin1} has a refinement in $M(I)$; on the other hand, it is not the case that $P_{\a\b}$ is true for every $\a,\b$, otherwise this would mean that $(x_{\a})_{i} = (y_{\b})_{i} = 0$ for all $\a,\b$ and $i>m$; thus $m$ would be a maximal element of $I$, contradicting our assumption. We will reach our goal by showing that, depending on the pairs $(\a,\b)$ for which $P_{\a\b}$ holds, there exist in $\BZ$ a refinement \eqref{eq refin1} such that $z_{\a\b}\in M(I)$ for all such pairs $(\a,\b)$.

Up to a possible exchange in \eqref{eq dimgrp4} between the summands of one or both members, there are three case to be considered.

\underline{Case 1}: $P_{\a\b}$ holds only for $(\a,\b)=(1,1)$, or only for $(\a,\b)=(1,1)$ \underline{and} $(\a,\b)=(2,2)$. Let us choose a refinement \eqref{eq refin1} where $a_{11}\ge 0$ and $a_{22}\ge 0$. If $a_{11}=0$, then $z_{11}=u_{11}\in M(J)\sbs M(I)$. If $a_{11}>0$, then $m$ is a maximal element of $\Supp(z_{11})$ and $(z_{11})_{m} = a_{11}>0$; if $j$ is a maximal element of $\Supp(z_{11})$ different from $m$ then, as seen previously, $(z_{11})_{j} = (u_{11})_{j}>0$. Thus $z_{11}\in M(I)$. A similar argument applies equally if $P_{\a\b}$ holds for $(\a,\b)=(1,1)$ and $(\a,\b)=(2,2)$.

\underline{Case 2}: $P_{\a\b}$ holds only for $(\a,\b)=(1,1)$ \underline{and} $(\a,\b)=(1,2)$. In this case, since $(x'_{1})_{i} = (u_{11})_{i}+(u_{12})_{i}$ for all $i$, if $(x_{1})_{m}\ne 0$, then $m$ must be a maximal element of $\Supp(x_{1})$, therefore $(x_{1})_{m}>0$. Let us choose the refinement \eqref{eq refin1} where
\[
\begin{pmatrix}
a_{11} & a_{21} \\
a_{12} & a_{22}
\end{pmatrix}
=
\begin{pmatrix}
(x_{1})_{m} & 0 \\
(x_{2})_{m}-(y_{2})_{m} & (y_{2})_{m}
\end{pmatrix}.
\]
If $(x_{1})_{m}=0$, then $(z_{11})_{m} = a_{11}=(x_{1})_{m} =0$ and so $z_{12}=u_{12}\in M(J)\sbs M(I)$. If $(x_{1})_{m}>0$, then $(z_{11})_{m} = a_{11}=(x_{1})_{m}>0$; consequently $(z_{11})_{m}m\in M(I)$ and hence $z_{11} = z'_{11}+ (z_{11})_{m}m\in M(I)$. Concerning $z_{12}$ we have that $(z_{12})_{m} = a_{12}=0$ and so $z_{12}=u_{12}\in M(J)\sbs M(I)$.

\underline{Case 3}: $P_{\a\b}$ holds for $(\a,\b)=(1,1)$, $(\a,\b)=(1,2)$, $(\a,\b)=(2,2)$ but not for $(\a,\b)=(2,1)$. As seen for Case 2, if $(x_{1})_{m}\ne 0$, then $m$ must be a maximal element of $\Supp(x_{1})$ and therefore $(x_{1})_{m}>0$. Similarly, since $(y'_{2})_{i} = (u_{12})_{i}+(u_{22})_{i}$ for all $i$, if $(y_{2})_{m}\ne 0$, then $m$ must be a maximal element of $\Supp(y_{2})$ and therefore $(y_{2})_{m}>0$. By choosing for \eqref{eq refin1} the same refinement we used in Case 2, we obtain again that $z_{11}$ and $z_{12}$ are in $M(I)$. By proceeding on $z_{22}$ in the same manner as we did in Case 2 for $z_{11}$, we see that $z_{22}\in M(I)$.

The proof that $M(I)$ is a refinement monoid is now complete.

Finally, without any assumption on the cardinality of $I$, suppose that $I$ has a finite cofinal subset $F$ and let $u$ be the element of $G(I)$ which is the sum of all elements in $F$. We observe that if $i\le j$ in $I$, given $a,b\in\BN$ with $a\le b$, then it is true that $ai\le bj$ and $-ai\le bj$, because both $bj-ai$ and $bj-(-ai) = bj+ai$ are in $M(I)$. As a first straightforward consequence, since the partial ordering of $G(I)$ is translation-invariant, if $v$ is any linear combination of elements of $F$, namely an element of $G(F)$, then $v\le au$ for a suitable positive integer $a$. Let $x\in G(I)$ and for every $i\in J$, let us choose an element $u_{(i)}\in F$ such that $i\le u_{(i)}$. By the above we have that
\[
x = \sum_{i\in J}x_{i}i \le \sum_{i\in J}|x_{i}|u_{(i)} \le au
\]
for a suitable positive integer $a$. This shows that $u$ is an order unit for $G(I)$.

Conversely, assume that $G(I)$ has an order unit $u$ and let $i\in I\setminus \Supp(u)$. Then there is a positive integer $a$ such that $x = au-i\in M(I)$. As a consequence $i$ is not a maximal element of $\Supp(x) = \Supp(u)\cup\{i\}$ and therefore $j>i$ for some $j\in\Supp(u)$. This shows that $\Supp(u)$ is a finite cofinal subset of $I$ and the proof is complete.
\end{proof}

\begin{re}
In \cite{AraGood:10} Ara and Goodearl say that an abelian monoid $M$ is \emph{tame} in case M is an inductive limit of some inductive system of finitely generated refinement monoids. Since our refinement monoid $M(I)$ is cancellative and unperforated, it follows from \cite[Theorem 2.14]{AraGood:10} that $M(I)$ is tame.
\end{re}

Assume now that $R$ is a semiartinian, unit-regular and very well behaved ring. Then the unit-regularity condition implies that $[A] = \overline{A}$ for all $A\in\FP_{R}$ and the canonical pre-order of $\KO(R)$ is a partial order: if $A,B\in\FP_{R}$, then $[A]\le [B]$ if and only if $A\lesssim B$.  By Theorem \ref{theo extendedtrace} we can choose an order-representative subset $\ZA$ in the class $\ESFP_{R}$. In turn, Proposition \ref{pro simpgenKO} tells us that $\KO(R)$ is free with $[\ZA]$ as a basis. By \eqref{eq order-selective} we have that $[\ZA]$, with the partial order induced by those of $\KO(R)$, is order isomorphic to $\Simp_R$ and hence to $\Prim_R$.

\begin{theo}\label{theo mainKO}
Let $R$ be a semiartinian, unit-regular, very well behaved ring and let $\ZA$ be an order-representative set $\ZA$ for the class $\ESFP_{R}$. Then
\[
\KO(R)^{+} = M([\ZA])
\]
and therefore $\KO(R)$ isomorphic, as a partially ordered abelian group, to $G(\Prim_R)$.
\end{theo}
\begin{proof}
Firstly we observe that $M([\ZA])$ is generated by all elements $[A]$ for $A\in \ZA$, together with all elements of the form
\[
[A]-n_{1}[B_{1}]-\cdots-n_{r}[B_{r}],
\]
where $n_{1},\ldots,n_{r}$ are positive integers, $A,B_{1},\ldots,B_{r}\in\ZA$ and $B_{i}\lesssim A$ for all $i\in\{1,\ldots,r\}$. If we consider an element of the second type, according to Corollary \ref{cor orderselect2} there is a suitable $A'\in\FP_{R}$ such that
\[
A\is A'\oplus B_{1}^{n_{1}}\oplus\cdots\oplus B_{r}^{n_{r}}.
\]
As a result we have in $\KO(R)$:
\[
[A] =  [A']+ n_{1}[B_{1}]+\cdots+ n_{r}[B_{r}]
\]
and hence
\[
[A] - n_{1}[B_{1}]-\cdots- n_{r}[B_{r}] = [A']\in \KO(R)^{+}.
\]
This shows that $M([\ZA])\sbs\KO(R)^{+}$.

Conversely, let $B\in\FP_{R}$. In order to prove that $[B]\in M([\ZA])$ we proceed by induction on the successor ordinal $h(B)$. If $h(B) = 0$, that is $B=0$, then there is nothing to prove. Thus, let $\a$ be any ordinal, assume that $[B]\in M([\ZA])$ whenever $h(B)\le\a$ and let us consider any $[B]\in M([\ZA])$ with $h(B)=\a+1$. Then there are distinct $A_{1},\ldots,A_{r}\in\ZA$, with $h(A_{i})=\a+1$ for all $i$, and positive integers $n_{1},\ldots,n_{r}$ such that
\[
B/BL_{\a} \is (A_{1}/A_{1}L_{\a})^{n_{1}}\oplus\cdots \oplus(A_{r}/A_{r}L_{\a})^{n_{r}}.
\]
Thus, if we set $A = A_{1}^{n_{1}}\oplus\cdots\oplus A_{r}^{n_{r}}$, we have that $B/BL_{\a} \is A/AL_{\a}$ and there are decompositions
\[
B = B'\oplus B'', \qquad A = A'\oplus A''
\]
such that
\[
B'\is A',\quad B'' = B''L_{\a} \quad \text{and}\quad A'' = A''L_{\a}.
\]
As a result we have
\[
[B] = [B']+[B''] = [A']+[B''] = [A]-[A'']+[B''];
\]
since $B''\in M([\ZA])$ by the inductive hypothesis, the proof will be complete once we have proved that $[A]-[A'']\in M([\ZA])$ as well. Since $[A]\in M([\ZA])$, we may assume that $A'' \ne 0$. Inasmuch as $[\ZA]$ is a basis for $\KO(R)$, we can write
\begin{equation}\label{eq order-selective4}
[A'']+c_{1}[C_{1}]+\cdots+c_{s}[C_{s}] = d_{1}[D_{1}]+\cdots+d_{t}[D_{t}],
\end{equation}
where $C_{1},\ldots,C_{s},D_{1},\ldots,D_{t}$ are \underline{distinct} elements of $[\ZA]$ and $c_{1},\ldots,c_{s},d_{1},\ldots,d_{t}$ are non-negative integers. If all $d_{j}$'s are zero, then $[A]-[A'']$ is a linear combinations of elements of $[\ZA]$ with non-negative coefficients and, in this case, we already know that $[A]-[A'']\in M([\ZA])$. Thus we may assume that in \eqref{eq order-selective4} all $d_{j}$'s are positive.

\underline{We claim that}
\[
    [C_{i}],[D_{j}]\in \{<\{[A_{1}],\ldots,[A_{r}]\}\}
\]
for every $j\in\{1,\ldots,t\}$ and for every $i\in\{1,\ldots,s\}$ such that $c_{i}>0$.
To this purpose, let us first prove that if $c_{i}>0$, for some $i$, then there is some $j\in\{1,\ldots,t\}$ such that $[C_{i}]<[D_{j}]$. In fact, if we rewrite \eqref{eq order-selective4} as
\begin{equation}\label{eq order-selective6}
    A''\oplus C_{1}^{c_{1}}\oplus\cdots\oplus C_{s}^{c_{s}} \is D_{1}^{d_{1}}\oplus\cdots\oplus D_{t}^{d_{t}},
\end{equation}
we infer that the simple module $C_{i}/C_{i}L_{h(C_{i})-1}$ imbeds into some $D_{j}/D_{j}L_{h(C_{i})-1}$, therefore $C_{i}/C_{i}L_{h(C_{i})-1}\preccurlyeq D_{j}/D_{j}L_{h(D_{j})-1}$ by the order-selectivity of $D_{j}$ and hence $C_{i}\lesssim D_{j}$ by \eqref{eq order-selective}. As $C_{i}\ne D_{j}$, we infer that $[C_{i}]<[D_{j}]$, as wanted. In order to prove our claim, let us show that, given $j_{1}\in\{1,\ldots,t\}$, we have $D_{j_{1}}\lesssim A_{k}$ for some $k\in\{1,\ldots,r\}$. Assume that it is not the case. Then the simple module $D_{j_{1}}/D_{j_{1}}L_{h(D_{j_{1}})-1}$ cannot imbed into $A/AL_{h(D_{j_{1}})-1}$ and hence does not imbed into $A''/A''L_{h(D_{j_{1}})-1}$. It follows from \eqref{eq order-selective6} that there is some $i_{1}\in\{1,\ldots,s\}$ for which $c_{i_{1}}>0$ and $D_{j_{1}}/D_{j_{1}}L_{h(D_{j_{1}})-1}$ imbeds into $C_{i_{1}}/C_{i_{1}}L_{h(D_{j_{1}})-1}$ and hence, by order-selectivity of $C_{i_{1}}$, we have that $[D_{j_{1}}]<[C_{i_{1}}]$. In turn, as we have proved before, there is some $j_{2}\in\{1,\ldots,t\}$ such that $[C_{i_{1}}]<[D_{j_{2}}]$. We have that  $[D_{j_{1}}]<[C_{i_{1}}]<[D_{j_{2}}]$, therefore for no $k\in\{1,\ldots,r\}$ it is true that $D_{j_{2}}\lesssim A_{k}$, otherwise we would get that $D_{j_{1}}\lesssim C_{i_{1}}\lesssim D_{j_{2}}\lesssim A_{k}$ and hence $D_{j_{1}}\lesssim A_{k}$, contradicting our assumption. By repeating the same argument as before, we infer that there is some $i_{2}\in\{1,\ldots,s\}$ such that $[D_{j_{2}}]<[C_{i_{2}}]$ and we get that $[D_{j_{1}}]<[C_{i_{1}}]<[D_{j_{2}}]<[C_{i_{2}}]$. It is then possible to build, by recursion, two countable sequences $(C_{i_{1}},C_{i_{2}},\ldots)$ and $(D_{j_{1}},D_{j_{2}},\ldots)$ of elements from the sets $\{C_{1},\ldots,C_{s}\}$ and $\{D_{1},\ldots,D_{t}\}$ respectively in such a way that
\[
[D_{j_{1}}]<[C_{i_{1}}]<[D_{j_{2}}]<[C_{i_{2}}]<\cdots.
\]
Since the range of the sequences is finite, in the above chain there are necessarily repetitions; this is not compatible with the fact that the relation $\le$ in $\KO(R)$ is a partial order. Thus our claim is established.

We have now
\[
[A]-[A''] = n_{1}[A_{1}]+\cdots+ n_{r}[A_{r}]- d_{1}[D_{1}]-\cdots-d_{t}[D_{t}]+ c_{1}[C_{1}]+\cdots+c_{s}[C_{s}].
\]
Since the elements $[A_{1}],\ldots,[A_{r}]$ form an antichain in $\ZA$, by the above they are the maximal elements of the set \[
\{[A_{1}],\ldots,[A_{r}],[D_{1}],\ldots,[D_{t}]\},
\]
therefore $n_{1}[A_{1}]+\cdots+ n_{r}[A_{r}]- d_{1}[D_{1}]-\cdots-d_{t}[D_{t}]\in M([\ZA])$. Since $c_{1}[C_{1}]+\cdots+c_{s}[C_{s}]\in M([\ZA])$, we conclude that $[A]-[A'']\in M([\ZA])$ and the proof is complete.
\end{proof}

As we had remarked in the introduction, if $R$ is any right semiartinian ring, then the partially ordered set $\Prim_A$ is artinian; moreover $\Prim_A$ has a finite cofinal subset in case $R$ is very well behaved (see \cite[Proposition 2.7]{Bac:16}). Conversely, given an artinian partially ordered set $I$ and a field $D$, in \cite{Bac:16} a construction is presented which produces a semiartinian and unit-regular $D$-algebra $D_{I}$, together with a family $(\zu_{i})_{i\in I}$ of idempotents satisfying the following properties:
\begin{enumerate}
\item $\zu_{i}D_{I}$ is eventually simple for every $i\in I$.
\item $\zu_{i}D_{I}\lesssim\zu_{i}D_{I}$ if and only if $i\le j$.
\item By setting $U_{i} = \zu_{i}D_{I}/\zu_{i}\Soc_{h(\zu_{i})-1}(D_{I})$, the assignment $i\mapsto U_{i}$ defines an injective map from $I$ to $\Simp_{D_I}$ and $U_{i}\preccurlyeq U_{j}$ if and only if $i\le j$.
\item If $\{m_{1},\ldots,m_{r}\}$ is a finite cofinal subset of $I$, then $\{\zu_{m_{1}},\ldots,\zu_{m_{r}}\}$ is a complete set of pairwise orthogonal idempotents of $D_{I}$ (see \cite[Proposition 7.1]{Bac:16}).
\end{enumerate}

It turns out that $D_I$ is very well behaved if and only if $I$ has a finite cofinal subset and, if it is the case, then $\ZA = \{\zu_{i}D_{I}\mid i\in I\}$ is an order-representative subset of $\ESFP_{D_{I}}$ and the correspondences $i\leftrightarrow \zu_{i}D_{I}\leftrightarrow U_{i}\leftrightarrow r_{R}(U_{i})$ define isomorphisms of partially ordered sets (see \cite[Theorem 9.5]{Bac:16}). Thus, if we consider the injective map $\map{f}{I}{\KO(D_{I})^{+}}$ defined by $f(i) = [\zu_{i}D_{I}]$ for all $i\in I$, in view of Theorem \ref{theo mainKO} $f$ extends to an isomorphism $\map{\overline{f}}{G(I)}{\KO(D_{I})}$ of partially ordered groups. According to Proposition \ref{pro unperforated}, $m_{1}+\cdots+m_{r}$ is an order-unit in $G(I)$ and $\overline{f}(m_{1}+\cdots+m_{r}) = [\zu_{m_{1}}D_{I}\oplus\cdots\oplus \zu_{m_{r}}D_{I}] = [D_{I}]$. Thus we may conclude with the following result:

\begin{theo}
Assume that $I$ is an artinian partially ordered set with a finite cofinal subset $F$ and let $u$ be the sum of all elements of $F$ in $G(I)$. Then there is an order-isomorphism from $(G(I),u)$ to $(\KO(D_I), [D_I])$, which restricts to an isomorphism from $M(I)$ to $\CV(D_I)$.
\end{theo}

\noindent\textbf{\emph{Acknowledgement}}. The authors wish to express their gratitude to Ken Goodearl, who noticed a crucial mistake in a preliminary version of this work. His counterexample gave a clear indication about the right direction to follow in order to correct that mistake. The correction, however, resulted in a complete rewriting of the third section.

\bibliographystyle{amsplain}

\end{document}